\documentclass[12pt]{amsart}

\usepackage{amssymb}
\usepackage{fancyhdr}
\usepackage{amsmath}
\usepackage{amsfonts}
\usepackage{mathrsfs}
\usepackage[all]{xy} % For commutative diagrams
\usepackage{color}

\newcommand{\disk}{\ensuremath{\mathbb{D}} } % unit disk
\newcommand{\sphere}{\bar{\Bbb{C}}} %Riemann sphere
\newcommand{\riem}{\Sigma}  %Riemann surface
\renewcommand{\Bbb}[1]{\ensuremath{\mathbb{#1}}}
\theoremstyle{plain}
        \newtheorem{theorem}{Theorem}[section]
        \newtheorem{lemma}[theorem]{Lemma}
        \newtheorem{proposition}[theorem]{Proposition}
        \newtheorem{corollary}[theorem]{Corollary}

\theoremstyle{definition}
        \newtheorem{definition}[theorem]{Definition}

\theoremstyle{remark}
    \newtheorem{remark}[theorem]{Remark}

\numberwithin{equation}{section} % Equation labels are 'section'.'eq #'
\numberwithin{figure}{section} % Figures labela are 'section.'fig #'

\usepackage{fullpage}

\author{Eric Schippers}
\author{Wolfgang Staubach}

\begin{document}

\title{Transmission of harmonic functions through quasicircles on compact Riemann surfaces}


\begin{abstract}
 Let $R$ be a compact surface and let $\Gamma$ be a Jordan curve which separates $R$ into two connected components $\riem_1$ and $\riem_2$.  A harmonic function $h_1$ on $\riem_1$ of bounded Dirichlet norm has boundary values $H$ in a certain conformally invariant non-tangential sense on $\Gamma$.  
 
 We show that if $\Gamma$ is a quasicircle, then there is a unique harmonic function $h_2$ of bounded Dirichlet norm on $\riem_2$ whose boundary values agree with those of $h_1$.  Furthermore, the resulting map from the Dirichlet space of $\riem_1$ into $\riem_2$ is bounded with respect to the Dirichlet semi-norm.   
\end{abstract}

\maketitle

\begin{section}{Introduction}

Let $R$ be a compact Riemann surface and let $\Gamma$ be a quasicircle separating $R$ into two complementary components $\riem_1$ and $\riem_2$. Given a function $h$ on $\Gamma$, do there exist elements of the Dirichlet space of $\riem_1$ and the Dirichlet space of $\riem_2$ that have $h$ as the boundary value? In this paper we show that $h$ is the boundary value of an element of the Dirichlet space of $\riem_1$ if and only if it is the boundary value of an element of the Dirichlet space of $\riem_2$. This leads us naturally to a concept that we refer to as the \textit{transmission} of a harmonic function from $\riem_1$ to $\riem_2$ through the quasicircle $\Gamma$.  The transmission of a Dirichlet bounded harmonic function on $\riem_1$ is obtained by first taking the boundary values on $\Gamma$, and then finding the corresponding Dirichlet bounded harmonic function on $\riem_2$ with these same boundary values. Our main result is that the resulting transmission map between the Dirichlet spaces of $\riem_1$ and $\riem_2$ is bounded (Theorem \ref{Thm:main boundedness of transmission}). \\

In the case when $\Gamma$ is a Jordan curve separating the Riemann sphere into two components, the authors showed that the transmission exists and is bounded if and only if $\Gamma$ is a quasicircle \cite{SchippersStaubach_jump}. Our proof here of the general case uses sewing techniques for Riemann surfaces.  Along the way, several other results are established which are of independent interest. For example, we show that a function on a Jordan curve is the boundary values of an element of the Dirichlet space of one of the complementary components $\riem$ if and only if it has an extension to a doubly connected neighbourhood, one of whose boundaries is the Jordan curve. Here, boundary values are obtained from a conformally invariant notion of non-tangential limit, and we show that boundary values exist
except on a set of capacity zero.  \\  

We define capacity zero sets along the boundary in terms of charts on doubly-connected neighbourhoods induced by Green's function.  It is possible to do this in greater generality \cite{RodinSario}, \cite{SarioNakai}, \cite{SarioOikawa}, for example in terms of the ideal boundary. However our approach is sufficiently general for the purposes of this paper, and leads as directly as possible to the transmission theorem.  

Similarly, we also required an extension of A. Beurling's theorem (concerning boundary values of harmonic functions with bounded Dirichlet energy) to Riemann surfaces with boundary.  The 
generalization of Beurling's theorem to Riemann surfaces is due to Z. Kuramochi, phrased in 
terms of what he calls N-fine limits \cite{Kuramochi}.  We give a proof using the charts induced by Green's 
function, in terms of a conformally invariant notion of non-tangential limit, in the case that the boundary is a Jordan curve in a compact surface.  This can conveniently be compared to Sobolev traces in the case that the boundary is sufficiently 
regular.  Again, we found this to be the most direct line to the proof of the transmission theorem on 
quasicircles.   We do not claim originality for the existence of one-sided boundary values, except 
perhaps for developing an approach which facilitates the application of sewing techniques, as it does in the present paper.  Indeed, the 
reader will recognize the ghost of the conformal welding theorem throughout the paper. \\

A central issue is that it must be shown that a set of capacity zero with respect to one side of the curve $\Gamma$ must have capacity zero with respect to the other.  We show that this holds (once again) for quasicircles.  Note that this is not true for harmonic measure: a set of harmonic measure zero with respect to one side of a quasicircle need not have harmonic measure zero with respect to the other (indeed, this fact prevents us from using several sources).  In particular, transmission is not even a well-posed problem for non Dirichlet-bounded harmonic functions, even though boundary values might exist up to a set of harmonic measure zero.  Further discussion can be found in \cite{RSS_Dirichlet_general}.\\

In the case of the sphere, the existence of the transmission is intimately connected to the Riemann boundary value problem (i.e the jump problem) on quasicircles, Faber-type approximations on the Dirichlet space and the Grunsky inequalities, and generalized period matrices  \cite{RSS_Dirichlet_general}, \cite{SchippersStaubach_jump}, \cite{SchippersStaubach_Grunsky_quasicircle}. The results of the present paper make the extension of 
these concepts to general Riemann surfaces possible, as will be shown in upcoming publications. \\

As is common practice, we will denote constants which can be determined by known parameters in a given situation, but whose value is not crucial to the problem at hand, by $C$.  The value of $C$ may differ from line to line, but in each instance could be estimated if necessary.

\end{section}
\begin{section}{Preliminaries}
\begin{subsection}{Green's functions and decompositions of harmonic functions}
 In this section we collect some well-known theorems and establish notation.  
 
If $R$ is a Riemann surface and $\riem \subset R$ is compactly contained in $R$, then we say 
that $g(z,w)$ is the Green's function for $\riem$ if $g(\cdot,w)$ is harmonic on $R \backslash \{w\}$,  $g(z,w) + \log{|\phi(z)-\phi(w)|}$ is harmonic in $z$
on a local parameter $\phi:U \rightarrow \mathbb{C}$ in an open neighbourhood $U$ of $w$, and $\lim_{z \rightarrow z_0} g(z,w) =0$ for all $z_0 \in \partial \riem$ and $w \in \riem$.  Green's function is unique and symmetric, provided it exists.  In this paper, we will consider only the 
case where $R$ is compact and no boundary component of $\riem$ reduces to a point, so Green's function of $\riem$ exists, see L. Ahlfors and L. Sario \cite[II.3]{Ahlfors_Sario}.  
 
Let $\ast$ 
denote the dual of the almost complex structure, that is
\[ \ast (a \, dx + b \,dy) = -b \,dx + a \,dy \]
for a one form $a \,dx + b \,dy$ written in local coordinates $z=x+iy$.  The one-form $\ast dg$ has a multi-valued primitive $\ast g$ which is locally a harmonic conjugate of $g$.   By Stokes' 
theorem, if $\riem$ is an open set, compactly contained in $R$ bounded by a single positively oriented 
simple closed curve $\Gamma$, then 
\begin{equation} \label{eq:Green_coperiod}
 \int_\Gamma \ast dg = - 2 \pi.  
\end{equation}
 
\begin{theorem}[H. L. Royden \cite{Royden}] \label{th:Rodin_harmonic_decomp}  Let $R$ be a compact Riemann surface.  Let $E\subseteq R$ be a closed
subset, and $O$ be an open set containing $E$.   Given a harmonic function $U$ in $O \backslash E$
there is a harmonic function $u$ in $R \backslash E$ such that $U-u$ extends harmonically to $O$ 
if and only if 
\[ \int_\Gamma \ast dU =0 \]
for some cycle $\Gamma$ homologous in $O \backslash E$ to the boundary of $O$. The function $u$ is 
unique up to a constant.  
\end{theorem}
We will require a special case of this theorem throughout this paper, which we state in the 
next section.
\end{subsection}
\begin{subsection}{Collar neighbourhoods}

 We will use the following terminology.  By a Jordan curve on a Riemann surface $R$, we mean a homeomorphic image of $\mathbb{S}^1$ in $R$.  

We use the following terminology for charts near a curve  $\Gamma$.    
 \begin{definition}  Let $\Gamma$ be a Jordan curve in a compact Riemann surface $R$.  
 \begin{enumerate}
 \item A collar neighbourhood of $\Gamma$ is an open connected set $A$ bordered by $\Gamma$ and $\Gamma'$, where $\Gamma'$ is a Jordan curve which is homotopic to $\Gamma$ and such that $\Gamma \cap \Gamma'$ is empty.  A collar chart for $\Gamma$ is a collar neighbourhood $A$ together with a biholomorphism $\phi:A \rightarrow \mathbb{A}$ where $\mathbb{A}$ is an annulus in $\mathbb{C}$.
 \item  A doubly-connected neighbourhood of $\Gamma$ is an open connected set $A$ containing $\Gamma$, bounded by two non-intersecting Jordan curves, both homotopic to $\Gamma$.  A doubly-connected
   chart for $\Gamma$ is a biholomorphism $\phi:A \rightarrow \mathbb{A}$ where $\mathbb{A}$ is an annulus in the plane $\mathbb{C}$ and $A$ is a doubly-connected neighbourhood of $\Gamma$.
 \end{enumerate}
\end{definition}
The condition that the image of $\phi$ is an annulus is not restrictive, since if the image were any doubly-connected domain, then one could compose with a conformal map onto an annulus \cite{Neharibook}. 
\begin{definition}
 We say that a Jordan curve $\Gamma$ in a Riemann surface is strip-cutting if it has a doubly-connected chart $\phi:A \rightarrow \mathbb{C}$.  
\end{definition}
A Jordan curve in $\mathbb{C}$ or $\sphere$ in the usual sense is strip-cutting.  Also, an analytic curve is by definition strip-cutting.   By shrinking the domain of the doubly-connected chart, one can always arrange that $\Gamma_1$ and $\Gamma_2$ are analytic curves in $R$ and $\phi$ is conformal on the closure of $A$.  We say that a Jordan curve separates $R$ if $R \backslash \Gamma$ has two connected components.  A strip-cutting Jordan curve need not be separating.  

\begin{remark}  If $\Gamma$ is a strip-cutting Jordan curve, then letting $B$ be one of the connected components of $A \backslash \Gamma$, it is easily seen that $\Gamma$ is a doubly-free boundary arc of $B$ in the Riemann surface $R$, see P. Gauthier and F. Sharifi \cite{GauthierSharifi}.   If we choose the boundaries of $A$ to also be strip-cutting, then $A$ and both components of $A \backslash \Gamma$ are then Jordan regions in $R$ in their sense.  

Furthermore, these domains are themselves bordered Riemann surfaces \cite[Theorem 5.1]{GauthierSharifi}.  Similarly, if a strip-cutting Jordan curves separates a compact surface $R$, then the components of the complement are bordered surfaces.  
\end{remark}

We will also need the following fact, which follows directly from standard results.
\begin{lemma}  \label{le:continuous_extension_collar}
 Let $\Gamma$ be a strip-cutting Jordan curve in a compact Riemann surface $R$.  Every collar chart $\phi:A \rightarrow \mathbb{A}$ extends continuously to a bijection from $A \cup \Gamma$ to 
 $\mathbb{A} \cup \gamma$ where $\gamma$ is one of the boundary circles of $\mathbb{A}$.  The restriction
 of the chart to $\Gamma$ is a homeomorphism onto the circle.
\end{lemma}
\begin{proof}
  Let $A_1$ and $A_2$ be domains in $\mathbb{C}$ bordered by non-intersecting homotopic Jordan curves.  Given a conformal map $f:A_1 \rightarrow A_2$, $f$ has a continuous, one-to-one extension from the closure of $A_1$ to the closure $A_2$, whose restriction to either boundary curve is a homeomorphism onto one of the boundary curves of $A_2$ \cite[Theorem 3.4, Sect 15.3]{ConwayII}.
 
 By the definition of strip-cutting Jordan curve, there is a map $\phi_0:A_0 \rightarrow A_1$ of 
 a collar neighbourhood $A_0$ to a doubly-connected domain $A_1$ bounded by two-non-intersecting homotopic Jordan curves, which has a homeomorphic extension taking $\Gamma$ to one of the curves.  Namely, one can take $\phi_0$ to be the restriction of a doubly-connected chart to one side of $\Gamma$.  Given 
 an arbitrary collar chart $\phi:A \rightarrow \mathbb{A}$, one may apply the previous paragraph 
 to $\phi \circ \phi_0^{-1}$ (possibly shrinking the domain), and thus 
 $\phi = (\phi \circ \phi_0^{-1}) \circ \phi_0$ has the desired homeomorphic extension.

 \end{proof}
 
 \begin{remark}  More general extension theorems for Jordan regions on Riemann surfaces can be found in \cite{GauthierSharifi}.   
 \end{remark}
 
A particular collar chart can be constructed using Green's functions in a standard way.  Let $\Gamma$ be a strip-cutting Jordan curve in a compact Riemann surface $R$, and assume that $\Gamma$ separates  $R$ into two components $\riem_1$ and $\riem_2$.  Choose one of the components $\riem$ and let $g$ denote its Green's function.  Let $A \subset \riem$ be some collar neighbourhood of $\Gamma$.

Fix a $p \in \riem$ and let $g_p(z)=g(z,p)$.  The one-form $\ast dg_p$ has a multi-valued primitive $\tilde{g}_p$ which is a local harmonic conjugate of $g_p$.  

Let $\gamma$ be a smooth curve in $A$ which is homotopic to $\Gamma$, and let 
$m = \int_\gamma \ast dg_p$.  Then the function 
\[  \phi = \exp{[-2 \pi  (g_p + i \tilde{g}_p)/m]}  \]
is holomorphic and single-valued on some region $A_r$ bounded by $\Gamma$ and a level curve
$\Gamma_r = \{ z\,:\, g(z,p) = r \}$ of $g_p$ for some $r >0$.  
A standard use of the argument principle show that $\phi$ is one-to-one 
and onto the annulus $\{ z : e^{-2\pi r/m} <|z|<1 \}$.  By Lemma 
\ref{le:continuous_extension_collar} $\phi$ has a continuous extension to 
$\Gamma$ which is a homeomorphism from $\Gamma$ onto its image. By increasing $r$, one can also arrange that $\phi$ extends analytically to $\Gamma_r$.

We call this collar chart the canonical collar chart with respect to $(\riem,p)$.  
It is uniquely determined up to a rotation and the choice of $r$ in the definition of domain.  The level sets $g(\cdot,p)=\epsilon$ are simple closed analytic curves for $\epsilon \in (0,r)$ for some $r>0$, but will not necessarily be simple or connected curves for general $\epsilon$.

There are of course many possible equivalent definitions of quasicircle; we fix one now.  In the plane $\mathbb{C}$, we say that a quasicircle is the image of $\mathbb{S}^1$ under a quasiconformal map of $\mathbb{C}$.  On a Riemann surface, we will use the following definition.
 \begin{definition}  A simple closed curve $\Gamma$ in a Riemann surface is a quasicircle if there 
 is a doubly-connected chart $\phi:A \rightarrow B \subset \mathbb{C}$ for $\Gamma$
 such that $\phi(\Gamma)$ is a quasicircle in $\mathbb{C}$.  
 \end{definition}
 In particular, a quasicircle is a strip-cutting Jordan curve.

\end{subsection}
\begin{subsection}{Dirichlet and Sobolev spaces} \label{se:Dirichlet_Sobolev}
In this section we establish notation, and collect some results on Sobolev spaces on Riemann surfaces. 

We define the Dirichlet spaces as follows.  Let $R$ be a compact Riemann surface, and let 
$\riem \subseteq R$ be open.  A function $f:\riem \rightarrow \mathbb{C}$ is harmonic if $f$ is $C^2$ and $d \ast d f=0$.
Equivalently, in a neighbourhood of every point, there is a local coordinate expression for $f$ so that $f$ is harmonic.  

The space of complex one-forms on $\riem$ has the natural inner-product 
\[ \left<\omega_1,\omega_2 \right>_{\riem} = \frac{1}{2} \iint_\riem \omega_1 \wedge \ast \overline{\omega_2}; \]
Denote by $L^2(\riem)$ the set of one-forms which are $L^2$ with respect to this inner product.
We define the harmonic Dirichlet space by 
\[ \mathcal{D}_\mathrm{harm}(\riem) = \{ f:\riem \rightarrow \mathbb{C} \,:\, 
  d \ast df =0 \ \text{and} \  df \in L^2(\riem) \}. \]

 We can define a degenerate inner product on $\mathcal{D}_{\mathrm{harm}}(\riem)$ by 
\[   (f,g)_{\mathcal{D}_{\text{harm}}(\riem)} = \left<df,dg \right>_{\riem}.   \]
 
If we define the operators 
\[   \partial f = \frac{\partial f}{\partial z} dz,  \ \ \ 
   \overline{\partial} f =  \frac{\partial f}{\partial \bar{z}} d \bar{z} \]
in local coordinates, then the inner product can be written as
\begin{equation} \label{eq:inner_product_with_dbar_and_d}
 (f,g)_{\mathcal{D}_{\text{harm}}(\riem)} = i \iint_\riem \left[ \partial f \wedge \overline{\partial g}  -  \overline{\partial} f \wedge 
 \partial \overline{g} \right].    
\end{equation}

Denote the holomorphic and anti-holomorphic Dirichlet spaces by
\begin{align*}
 \mathcal{D}(\riem) &  =  \{ f \in \mathcal{D}_\mathrm{harm}(\riem) \,:\, \overline{\partial}f =0 \} \\
 \overline{\mathcal{D}(\riem)} & = \{ f \in \mathcal{D}_\mathrm{harm}(\riem) \,:\, {\partial}f =0 \}.
\end{align*}
Of course $\overline{\mathcal{D}(\riem)} = \{ \overline{f}: f \in \mathcal{D}(\riem) \}$ so the notation is consistent.  These two spaces are orthogonal with respect to the inner product given by (\ref{eq:inner_product_with_dbar_and_d}). \\ 

Next we gather some standard results from the theory of Sobolev spaces which we shall use in this paper.
For more specific details of the discussions below, as well as theory of Sobolev spaces and elliptic boundary value problems on manifolds, we refer the reader to \cite{Taylor}, and for the theory of Sobolev spaces in domains of $\mathbb{R}^n$ we refer to \cite{Mazya}.

Let $(M, h)$ be a compact Riemann surface endowed with a hyperbolic metric $h$ and $f$ a function defined on $M$. 
Set $d\sigma(h):= \sqrt{|\det h_{ij}|}\,|dz|^2$ which is the area-element of $R$, where $h_{ij}$
are the components of the metric with respect to coordinates $z = x_1 + i x_2$.  We define the inhomogeneous and homogeneous Sobolev norms and semi-norms respectively of $f$ as

\begin{equation}\label{defn: Sobolev norm on riem}
\Vert f\Vert_{H^1 (R)}:= \Big(\iint_R |f|^2 d\sigma(h)\Big)^{\frac{1}{2}}+ \Big(\iint_R df \wedge \ast \overline{df}\Big)^{\frac{1}{2}},
\end{equation}
and
\begin{equation}\label{defn: hom Sobolev norm on riem}
\Vert f\Vert_{\dot{H}^1 (R)}:=\Big(\iint_R df \wedge \ast \overline{df}\Big)^{\frac{1}{2}},
\end{equation}
thus the Dirichlet semi-norm and the homogeneous Sobolev semi-norm are given by the same expression. 

Observe that since any two metrics on $R$ has comparable determinants, choosing different metrics in the definitions above yield equivalent norms.  When necessary to specify the underlying metric, we will use the notation $H^1(R,h)$.   Now if $R$ is a compact Riemann surface and $\riem$ is an open subset of $R$ with analytic boundary $\partial \riem$ then the pull back of the metric $h_{ij}$ under the inclusion map yields a metric on $\riem$. Using that metric, we can define the inhomogeneous and homogeneous Sobolev spaces $H^1(\Sigma)$ and $\dot{H}^1(\Sigma)$. However these definitions will a-priori depend on the choice of the metric induced by $R$, due to the non-compactness of $\riem$, unless further conditions on $\riem$
are specified.  \\

We will also use the fractional Sobolev $H^{\frac{1}{2}}(M)$ on a compact smooth
n-dimensional manifold $M$,  whose definition we now recall.

First,  we define the Sobolev space $H^{\frac{1}{2}}(\mathbb{R}^n),$  which consists of tempered distributions $u$ such that $\int_{\mathbb{R}^n} (1+|\xi|^2)^{1/2}|\widehat{u}(\xi)|^2 \, d\xi<\infty$, where $\widehat{u}(\xi)$ is the Fourier transform of $u$. Now, let $M$ be an $n-$dimensional smooth compact manifold without boundary. Given a distribution $f\in \mathscr{D}'(M)$ (see e.g. \cite{Taylor} for details) one says that $f\in H^{1/2}(M),$ if for any coordinate patch $(\phi,U)$ of $M$ and any $\psi\in C^{\infty}_c(U)$,  $\psi f|_U \circ \phi^{-1}\in H^{1/2}(\mathbb{R}^n).$ 

To define the Sobolev norms, one starts with the smooth atlas $(\phi_j, U_j)$ and the corresponding smooth partition of unity $\psi_j$ with $\psi_j\geq 0$, $\mathrm{supp}\, \psi_j \subset U_j$ and $\sum_j \psi_j =1$, and then one defines

\begin{equation}\label{defn: Sobolev norm}
\Vert f\Vert_{H^{\frac{1}{2}} (M)}:=\sum_j \Vert (\psi_j f)\circ \phi^{-1}_j \Vert_{H^{\frac{1}{2}}(\mathbb{R}^n)}
\end{equation}
As is  well-known, different choices of the atlas and its corresponding partition of unity, produces norms that are equivalent with \eqref{defn: Sobolev norm}, see e.g. \cite{Booss}, Proposition 11.2 on page 68.

A rather fundamental fact in the study of boundary value problems is the Sobolev trace theorem that asserts that for $f\in H^1 (\riem)$ the restriction of $f$ to the smooth boundary $\partial \Sigma$ is in $H^{1/2}(\partial \Sigma)$, and the trace operator is a bounded linear map, see e.g. \cite[Proposition 4.5 p 334]{Taylor}. 
\begin{remark}
It is important to note that in all cases we consider the $H^{1/2}$ space of $\partial \Sigma$ in this paper, we assume that $\riem \subset R$ for a compact surface $R$ and that $\partial \riem$ is an analytic curve (and in particular smooth) and thus an embedded submanifold of $R$.  Thus the charts on  $\partial \riem$ can be taken to be restrictions of charts from $R$.  For roughly bounded $\riem$, the boundary $\partial \riem$ can be endowed with the manifold structure obtained by treating it as the ideal boundary of $\riem$. However, we will not apply the Sobolev theory directly to such domains.  Indeed in those cases the boundary is of course not a submanifold of $R$.    
\end{remark}
 
 \begin{proposition}\label{equiv and sobolevs in the plane}
Let $\Omega$ be a bounded domain in $\mathbb{R}^n$ with Lipschitz boundary, then $H^1(\Omega)=\dot{H}^1(\Omega)$ as sets. 
\end{proposition}
\begin{proof}
This is a special case of a more general result (which is even valid for domains whose boundaries are unions of continuous graphs), which is given in \cite[Corollary, p21]{Mazya} and the remark that follows that corollary.
It is also explicitly stated for Sobolev spaces more general than just $H^1(\Omega)$ in Proposition 1.25.2 in \cite{Medkova}.
\end{proof} 
In light of this result we have
\begin{theorem}\label{equiv and sobolevs in riem}
 Let $R$ be a compact surface and let $\riem \subset R$ be bounded by a closed analytic curve $\Gamma$.  Fix a Riemannian metric $\Lambda_R$ on $R$ as follows.   If $R$ has genus $g>1$ then $\Lambda_R$ is the hyperbolic metric; if $R$ has genus $1$ then $\Lambda_R$ is the Euclidean metric, and if $R$ has genus $0$ then $\Lambda_R$ is a spherical metric.   
 Let 
 $H^1(\riem)$ and $\dot{H}^1(\riem)$ denote the Sobolev spaces with respect to $\Lambda_R$.  Then $\dot{H}^1(\riem) = H^1(\riem)$ as sets. 
\end{theorem}
\begin{proof} 
It is of course true that $H^1(\riem)\subset \dot{H}^1(\riem)$ and $\Vert \cdot\Vert_{\dot{H}^1(\riem)}\leq \Vert \cdot\Vert_{H^1(\riem)}.$ Therefore it remains to show the non-trivial continuous inclusion $\dot{H}^1(\riem) \subset H^1(\riem).$\\
We deal with the most complicated case first.  Assume that $R$ is hyperbolic.   Let $\Pi:\disk \rightarrow R$ be the holomorphic universal covering map. Fix $p \in \riem$ and  choose closed curves $\gamma_1,\ldots,\gamma_{2g}$ based at $p$ generating the fundamental group
of $R$ based at $p$.  Homotopic representatives can be chosen so that (1) the curves do not intersect except at $p$, (2) the curves are smooth, except possibly at $p$, where they meet at non-zero angles mod $2 \pi$, and (3) each curve intersects $\Gamma$ at most finitely many times.  The last condition can be arranged because $\Gamma$ is a compact analytic curve.  Let $\gamma = \gamma_1 \cup \cdots \cup \gamma_{2g}$.  

By cutting $R$ along these curves,
one obtains an open fundamental region $D \subset \disk$ such that $\Pi$ is injective on $D$. Moreover, $\Pi$ maps the closure of $D$ onto $R$, and the boundary of $D$ onto $\gamma$.  
Since $R$ is compact the fundamental region is compactly contained in $\mathbb{D}$. The conditions on $\gamma_i$ ensure that $\Omega=D \cap \Pi^{-1}(\riem \backslash \gamma)$ is an open set bounded by $n$ piecewise smooth curves, meeting at non-zero angles at the vertices.  
If $\gamma$ does not intersect $\Gamma$, then $\Omega$ has two boundary components one of which is $\Pi(\gamma)$ and the other is $\Pi(\Gamma)$; otherwise, $\partial \Omega$ is connected.  

Now the pull-back of $\Lambda_R$ under $\Pi$ is equal to the Poincar\'e metric $\lambda(z)^2 |dz|^2$ on $\disk$.  Thus since $\gamma \cup \Gamma$ has measure zero in $R$, for any $h \in H^1(\riem)$ we 
have
\begin{equation} \label{eq:hyp_comparison_temp1}
  \| h \|_{H^1(\riem,\Lambda_R)} \sim \| h \circ \Pi \|_{H^1(\Omega,\lambda(z)|dz|^2)} 
\end{equation}
and 
\begin{equation}  \label{eq:hyp_comparison_temp2}
 \| h \|_{\dot{H}^1(\riem,\Lambda_R)} \sim \| h \circ \Pi \|_{\dot{H}^1(\Omega,\lambda(z)|dz|^2)} 
\end{equation}

Since the closure of $\Omega$ is contained in $\disk$, there is a constant $C>1$ such that 
\[  1/C \leq \lambda(z) \leq C   \] 
uniformly on $\Omega$; thus the Poincar\'e metric is comparable to the Euclidean on $\Omega$ and so
\begin{equation} \label{eq:hyp_comparison_temp3}
  \| h \circ \Pi \|_{H^1(\Omega,\lambda(z)|dz|^2)}  \sim  \| h \circ \Pi \|_{H^1(\Omega,|dz|^2)} 
\end{equation}
and 
\begin{equation} \label{eq:hyp_comparison_temp4}
  \| h \circ \Pi \|_{\dot{H}^1(\Omega,\lambda(z)|dz|^2)}  \sim  \| h \circ \Pi \|_{\dot{H}^1(\Omega,|dz|^2)}.
\end{equation}

Now by Proposition \ref{equiv and sobolevs in the plane}, since $\Omega$ is a bounded domain in $\mathbb{R}^2$ with Lipschitz boundary,
we have that $\dot{H}^1(\Omega,|dz|^2)= H^1(\Omega,|dz|^2)$ as sets.

Combining this with (\ref{eq:hyp_comparison_temp1}), (\ref{eq:hyp_comparison_temp2}), (\ref{eq:hyp_comparison_temp3}), and (\ref{eq:hyp_comparison_temp4}) completes the proof in the case 
that $g >1$.   

If $g=1$, then the argument proceeds in the same way, except that the covering map directly induces an equivalence between the Euclidean metric on $R$ and the Euclidean metric on the plane, so estimates of the form (\ref{eq:hyp_comparison_temp3}) and (\ref{eq:hyp_comparison_temp4}) are trivial. 

If $g=0$, then $R$ is conformally equivalent to the Riemann sphere and isometric to the sphere with the spherical metric $|dz|^2/(1+|z|^2)^2$ (up to a scale factor).  Thus we assume that $\Lambda_R= K|dz|^2/(1+|z|^2)^2$ for some constant $K$.  Since the fundamental group is trivial we may obviously 
omit the step of cutting along curves.  Applying an isometry of the spherical metric, one can assume that $\riem \cup \Gamma$ does not contain the point at $\infty$; that is, $\riem$ is a bounded subset of the plane
bordered by an analytic curve.  One now proceeds as in the hyperbolic case, using the fact that the spherical metric satisfies 
\[  1/C \leq \frac{K}{(1+|z|^2)^2} \leq C \]
on $\riem$.   
\end{proof}
\begin{remark}
Theorem \ref{equiv and sobolevs in the plane} is also true for more general Sobolev spaces $H^k(\riem)$ where $k\in \mathbb{N}$.
\end{remark}

In the proof of the boundedness of the transmission operator on the Dirichlet spaces (Theorem \ref{Thm:main boundedness of transmission}),  we will also need  the following estimate for Sobolev norms of harmonic extensions, which is a consequence of the Dirichlet principal.
 \begin{lemma}\label{Dirichlets princip} 
   Let $\riem$ be a Riemann surface bounded by an analytic curve $\Gamma$ $($in the sense that $\riem \subset R$ 
   where $R$ is a Riemann surface and the boundary of $\riem$ is an analytic curve in $R$$)$.
   There is a constant $C$ so that for any $u,v \in H^1(\riem)$ with the same trace in $H^{1/2}(\Gamma)$ almost everywhere, such that $u$ is harmonic,  $\| u \|_{\dot{H}^1 (\riem)} \leq 
   \| v \|_{\dot{H}^1 (\riem)}$.  
 \end{lemma}
 
 \begin{proof}  
This can be proved in a similar way as in the case of domains in $\mathbb{R}^n$ with smooth boundary which relies on two ingredients, namely Rellich's theorem and the Poincar\'e's inequality (see i.e. \cite[page 419-420 section 1.2]{Helein}). Thus, for manifolds with smooth boundary, the proof proceeds just as in the Euclidean case, where one replaces the aforementioned theorems with the ones valid on manifolds with smooth boundary. The Rellich theorem in this setting is given in \cite{Taylor} page 334 Proposition 4.4, and the validity of the Poincar\'e inequality is a consequence of the boundedness of $\riem$ (with respect to the metric on $R$).
\end{proof}  

Finally, we need the existence of decomposition of Dirichlet-bounded harmonic functions on doubly-connected regions.  First, we need a lemma.
\begin{lemma}  \label{le:Greens_bounded_near_boundary}  Let $\Gamma$ be a strip-cutting Jordan curve in a 
 compact Riemann surface $R$, bounding a Riemann surface $\riem$ with Green's function $g$.  Let $p \in \riem$.  There is a collar neighbourhood $A$ of $\Gamma$, such that $p \notin A$ and 
 \[ (g_p,g_p)_A =  \iint_A dg_p \wedge \ast dg_p < \infty, \]
 \end{lemma}
 where $g_p:= g(\cdot,p)$.
 \begin{proof}  Let $\Gamma_\epsilon = \{z \,:\, g(z,p) = \epsilon \}$ be a level curve of Green's function.  For $s$ sufficiently small, we obtain a collar neighbourhood bounded 
 by $\Gamma_s$ and $\Gamma$.  Let $A$ be the domain bordered by $\Gamma$ and $\Gamma_s$.  
 
 For any $r \in (0,s)$, $\Gamma_r$ is a simple closed analytic curve, and 
 \[ \int_{\Gamma_r} \ast dg_p = c \]
 where $c$ is some constant independent of $r$.  Thus by Stokes' theorem we see that
 \begin{align*}
 (dg_p,dg_p)_A & = \iint_{A} dg_p \wedge \ast dg_p = \lim_{\epsilon \rightarrow 0} \int_{\Gamma_\epsilon}
    g_p \ast dg_p - \int_{\Gamma_s} g_p \ast dg_p \\
    & = \lim_{\epsilon \rightarrow 0}  \epsilon \int_{\Gamma_{\epsilon}} \ast dg_p 
    -s \int_{\Gamma_s} \ast dg_p \\
    & = \lim_{\epsilon \rightarrow 0}(\epsilon-s)c=- cs.   
 \end{align*}
Thus $c$ is negative real and $(dg_p,dg_p)_A < \infty$.  
 \end{proof}
 
 \begin{corollary}  \label{co:harm_in_out_double_decomp} Let $R$ be a compact Riemann surface and let $A$ be a doubly-connected domain in $R$
 bounded by strip-cutting Jordan curves $\Gamma_1$ and $\Gamma_2$.
  For $i=1,2$ let $B_i$ be the component of the complement of $\Gamma_i$ containing $A$.
   Let $p_1$ be a point in $B_1 \backslash \mathrm{cl} A$ $($$\mathrm{cl}$ denotes the closure$).$ 
Let $g_{p_1}(w)$ be Green's function on $B_1$ with logarithmic singularity at $p_1$.  Every $h \in \mathcal{D}(A)$ can be written as
 \[  h = h_1 + h_2 + c g_{p_1}  \]
 where $h_1 \in \mathcal{D}_{\mathrm{harm}}(B_1)$ and $h_2 \in \mathcal{D}_{\mathrm{harm}}(B_2)$
 and $c  \in \mathbb{C}$.
 \end{corollary}
 
 \begin{proof}  Choose a smooth curve $\Gamma$ in $A$  
 which is homotopic to $\Gamma_1$ and $\Gamma_2$.  Using equation (\ref{eq:Green_coperiod}), 
  there is some constant $c \in \mathbb{C}$ such that $h - c g_{p_1}$ satisfies 
  \[  \int_\Gamma \ast d(h-cg_{p_1}) =0.  \]
  Applying Theorem \ref{th:Rodin_harmonic_decomp} to $h- c g_{p_1}$ shows that 
  $h= h_1 + h_2 + c g_{p_1}$ for some $h_i$ harmonic in $B_i$ for $i=1,2$.  For any such decomposition, we have by Stokes' theorem 
  \[    \int_\Gamma  \ast d (h_1 + h_2) = 0,  \] 
  so the $c$ is uniquely determined. So if $H_1 + H_2 + C g_{p_1}$ is any other decomposition, then $c=C$, and by the uniqueness statement in Theorem \ref{th:Rodin_harmonic_decomp} $H_1 = h_1 + e$ and $H_2 = h_2 -e$ for some constant $e$.  
  
  Now let $\Gamma'$ be a curve in $A$ homotopic to $\Gamma_1$.  Let $A' \subset A$ be the 
  collar neighbourhood of $\Gamma_1$ bounded by $\Gamma_1$ and $\Gamma'$.  Choose $\Gamma'$ 
  so that $(dg_{p_1},dg_{p_1})_{A'} < \infty$; this can be done by Lemma \ref{le:Greens_bounded_near_boundary}.  
 Since $h_2$ is harmonic on an open neighbourhood of the closure of $A'$, $h_2 \in \mathcal{D}(A')_\mathrm{harm}$.
  Also, since $A' \subset A$ and $h \in \mathcal{D}_\mathrm{harm}(A)$, it follows that $h \in \mathcal{D}(A')_\mathrm{harm}$.  Therefore $h_1= h - h_2 - c g_{p_1} \in \mathcal{D}(A')_\mathrm{harm}$.   Since 
  $h_1$ is harmonic on an open neighbourhood $V$ of the closed set $B_1 \backslash A'$, it thus is also in $\mathcal{D}(U)_{\text{harm}}$ for some open set $U \subset V$ containing $B_1 \backslash A'$.  Thus   
  $h_1 \in \mathcal{D}_\mathrm{harm}(B_1)$.  A similar argument shows that $h_2 \in \mathcal{D}_\mathrm{harm}(B_2)$.  
\end{proof}  
\end{subsection}
\end{section}
\begin{section}{Transmission of harmonic functions}
 In this section we derive the main results.

\begin{subsection}{Null sets}
We need a notion of zero-capacity set on the boundary of Riemann surfaces $\riem$ contained
in a compact surface $R$.  For our purposes it suffices to consider the case that the 
boundary is a strip-cutting Jordan curve.
\begin{definition}  Let $\Gamma$ be a strip-cutting Jordan curve in a compact Riemann 
 surface $R$, which separates $R$ into two connected components.  Let $\riem$ be one 
 of the components, and let $\phi$ be the canonical collar chart
 with respect to $(\riem,p)$.  
 We say that a Borel set $I \subset \Gamma$ 
 is null with respect to $(\riem,p)$ if $\phi(I)$ has logarithmic capacity
 zero in $\mathbb{C}$.  
\end{definition}

This is independent of $p$, and of the choice of collar chart. 
In the statement of the theorem below, recall that by Lemma \ref{le:continuous_extension_collar} every collar 
chart has a continuous extension to $\Gamma$.  
\begin{theorem}  \label{th:null_sets_both_sides} Let $\Gamma$ be a strip-cutting Jordan curve
in a compact Riemann surface $R$ 
 which separates $R$ into two connected components.  Fix one of these and denote it by $\riem$.  Let 
 $I \subset \Gamma$ be a Borel set.  The following are equivalent.
 \begin{enumerate}
  \item $I$ is null with respect to $(\riem,q)$ for some $q \in \riem$;
  \item $I$ is null with respect to $(\riem,q)$ for all $q \in \riem$;
  \item There exists a collar chart $\phi:A \rightarrow \mathbb{A}$ for $A \subseteq \riem$
   such that $\phi(I)$ has logarithmic capacity zero;
  \item $\phi(I)$ has logarithmic capacity zero for 
   every collar chart $\phi:A \rightarrow \mathbb{A}$ for $A \subseteq \riem$.
 \end{enumerate}
\end{theorem}
\begin{proof}
 It is clear that (2) implies (1), (1) implies (3), and (4) implies (2).  We will show
 that (3) implies (4).  
 
 If $K \subset \mathbb{S}^1 = \{z\,:\, |z|=1 \}$ is 
 a Borel set of logarithmic capacity zero, and $\phi$ is a quasisymmetry, then $\phi(K)$ has logarithmic capacity zero { \cite[Theorem 2.9]{SchippersStaubach_jump}, ultimately relying on \cite{ArcozziRochberg}}.  Since the inverse of a quasisymmetric map is also a quasisymmetry (and in particular 
 a homeomorphism), we see that a Borel set $K$ has logarithmic capacity zero if and only if $\phi(K)$  is a Borel set of logarithmic capacity zero. 
 
 Now let $\phi:A \rightarrow \mathbb{A}$ and $\psi:B \rightarrow \mathbb{B}$ be collar charts
  such that $A$ and $B$ are in $\riem$.   By 
 composing with a scaling and translation we can obtain maps $\tilde{\phi}$ and $\tilde{\psi}$ such that the image of $\Gamma$ under both 
 $\tilde{\phi}$ and $\tilde{\psi}$ is $\mathbb{S}^1$; we can also arrange that $\mathbb{S}^1$ is the outer boundary of both $\mathbb{A}$ and $\mathbb{B}$ by composing with $1/z$ if necessary.   By Lemma \ref{le:continuous_extension_collar}, $\tilde{\phi}\circ \tilde{\psi}^{-1}$  has a homeomorphic extension to $\mathbb{S}^1$.  By Schwarz reflection principle, it has an extension to a conformal map of an open neighbourhood of $\mathbb{S}^1$, so it is an analytic diffeomorphism of $\mathbb{S}^1$ and in particular a 
 quasisymmetry. Thus $\tilde{\psi}(I)$ has logarithmic capacity zero if and only if $\tilde{\phi}(I)$ has capacity zero.  Since linear maps $z \mapsto az  + b$ take  
 Borel sets of capacity zero to Borel sets of capacity zero, as does $z \mapsto 1/z$, we have that $\phi(I)$ has logarithmic 
 capacity zero if and only if $\psi(I)$ does.  This completes the proof.
\end{proof}

Thus, we now see that we may say that ``$I$ is null with respect to $\riem$'' without ambiguity.

This result can be improved for quasicircles to allow the possibility that the collar charts are in different 
components.
\begin{theorem} \label{th:quasicircle_null_eitherside}
 Let $\Gamma$ be a quasicircle in a compact Riemann surface $R$, separating $R$ into two
 connected components $\riem_1$ and $\riem_2$.  Let $I \subset \Gamma$ be Borel.  Then $I$ is null
 with respect to $\riem_1$ if and only if $I$ is null with respect to $\riem_2$.
\end{theorem}
\begin{proof}  Let $g:U \rightarrow V$ be a doubly-connected chart in a neighbourhood of $\Gamma$. 
 By shrinking $U$ if necessary, we can assume that $U$ is bounded by analytic curves $\gamma_1$ 
 and $\gamma_2$ in $\riem_1$ and $\riem_2$ respectively, and 
 that $g$ has a conformal extension to an open set containing the closure of $U$ so that $g(\gamma_1)$
 and $g(\gamma_2)$ are analytic curves in $\mathbb{C}$.  
 Let $\phi:A \rightarrow \mathbb{A}$ be a collar chart in a neighbourhood of $\Gamma$ in $\riem_1$
 and $\psi:B \rightarrow \mathbb{B}$ be a collar chart in a neighbourhood of $\Gamma$ in $\riem_2$.
 For definiteness, we arrange that the outer boundary of both annuli $\mathbb{A}$ and $\mathbb{B}$
 is $\mathbb{S}^1$, and that $\phi$ and $\psi$ both take $\Gamma$ to $\mathbb{S}^1$.  This can 
 be done by composing with an affine transformation and $z \mapsto 1/z$ if necessary.   Let $\Omega^+$
 denote the bounded component of the complement of $g(\Gamma)$ in $\sphere$ and $\Omega^-$ denote
 the unbounded component.  We assume that $g$ takes $U \cap \riem_1$ into $\Omega^+$, again by 
 composing with $z \mapsto 1/z$ if necessary.  Finally, by possibly shrinking the domain of $g$ again,
 we can assume that the analytic curve $\gamma_1$ is contained in the domain of $\phi$.  
 
 Thus, $\phi \circ g^{-1}$ is a conformal map of a collar neighbourhood $W$ of $g(\Gamma)$ in $\Omega^+$
 onto a collar neighbourhood of 
 $\mathbb{S}^1$ in $\disk$, whose inner boundary $\phi(\gamma_1)$ is an analytic curve.
 By the previous paragraph it has a conformal extension to an open neighbourhood of $g(\gamma_1)$, and thus 
 the restriction of $\phi \circ g^{-1}$ is an analytic diffeomorphism from $g(\gamma_1)$ to $\phi(\gamma_1)$.
 Thus if we let $W'$ be the Borel set in $\Omega^+$ bounded by $g(\gamma_1)$, there is a quasiconformal 
 map $F$ of $W'$ with a homeomorphic extension to $g(\gamma_1)$ equalling $\psi \circ g^{-1}$.   The 
 map 
 \[  \Phi(z) = \left\{ \begin{array}{cc} F(z) & z \in W' \\ \phi \circ g^{-1}(z) & z \in W \cup 
     g(\gamma_1)   \end{array} \right.  \]
 is therefore a quasiconformal map from $\Omega^+$ to $\disk$.   A similar argument 
 shows that
 $\psi \circ g^{-1}$ has a quasiconformal extension to a  map from $\Omega^-$ to $\disk$.  
 
 Since $g(\Gamma)$ is a quasicircle, there is a quasiconformal reflection $r$ of the plane which fixes each point in $g(\Gamma)$.  Thus $\psi \circ g^{-1} \circ r \circ (\phi \circ g^{-1})^{-1}$ has an extension to an (orientation reversing) quasiconformal self-map of the disk.  Thus it extends continuously to a quasisymmetry of $\mathbb{S}^1$, which takes Borel sets of capacity zero to Borel sets of capacity zero. Furthermore, on $\mathbb{S}^1$, this map equals $\psi \circ \phi^{-1}$.  Since the same argument applies to $\phi \circ \psi^{-1}$, 
 we have shown that $\phi(I)$ has capacity zero in $\mathbb{S}^1$ if and only if $\psi(I)$ has capacity zero
 in $\mathbb{S}^1$. This completes the proof.  
\end{proof}

 Let $\Gamma$ be a strip-cutting Jordan curve in $R$.  We say that a property holds quasi-everywhere on $\Gamma$ with respect to $\riem$ if it holds except on a null set with respect to $\riem$.
 Thus we define the following class of functions.  
 \[  \mathcal{B}(\Gamma;\riem) = \{ h:\Gamma \rightarrow \mathbb{C}   \}/\sim  \]
 where $h_1 \sim h_2$ if $h_1 = h_2$ quasi-everywhere with respect to $\riem$.  By Theorem \ref{th:quasicircle_null_eitherside}, if 
 $\Gamma$ is a quasicircle separating $R$ into connected components $\riem_1$ and 
 $\riem_2$, then 
 $\mathcal{B}(\Gamma,\riem_1) = \mathcal{B}(\Gamma,\riem_2)$.  Thus we will use the 
 notation $\mathcal{B}(\Gamma)$ in this case.  
\end{subsection}
\begin{subsection}{The class $\mathcal{H}(\Gamma)$ for simply connected domains}

 We will need a notion of non-tangential limit for points on the boundary of a Riemann surface $\riem$ contained in a compact surface $R$, when the boundary of $\riem$ is a strip-cutting Jordan curve.  This notion is designed to be conformally invariant, and thus not depend on the regularity of the boundary.  Indeed, it could be phrased in terms of the ideal boundary, but we will not require this or pursue it beyond making a few remarks.
 
 We begin with simply-connected domains.  
 A non-tangential wedge in $\disk$ with vertex at $p \in \mathbb{S}^1$ is a set of the form 
 \[ W(p,M)  = \{ z \in \disk : |p-z| < M(1-|z|)   \} \]
 for some $M \in (1,\infty)$.  
 As usual, we say that a function $f:\disk \rightarrow \mathbb{C}$ has a non-tangential limit at $p$
 if the limit of $\left. f \right|_{W(p,M)}$ as $z \rightarrow p$ exists for all $M \in (1,\infty)$.  
 One may of course equivalently use Stolz angles, that is sets of the form 
 \[ S(p, \alpha)= \{z: \text{arg}(1-\bar{p} z) < \alpha, |z - p| < \rho  \}   \]
 where $\alpha \in (0,\pi/2)$ and $\rho < \cos{\alpha}$ \cite[p6]{Pommerenke_boundary_behaviour}.  It is easily seen that if $T:\disk \rightarrow \disk$ is a disk automorphism, then $f$ has a non-tangential limit at $p$ if and only if $f \circ T$ has a non-tangential limit at $T(p)$.    
 
 \begin{definition} Let $\Omega \subset R$ be a simply-connected domain in a compact Riemann surface which is bordered by a 
 strip-cutting Jordan curve.  
 We say that a function $h:\Omega \rightarrow \mathbb{C}$ has a conformally non-tangential limit at $p \in \partial \Omega$ if and only if $h \circ f$ has non-tangential limit at $f^{-1}(p)$ for a conformal map $f:\disk \rightarrow \Omega$.  
 \end{definition}
 Observe that by Carath\'eodory's theorem $f$ has a homeomorphic extension taking $\mathbb{S}^1$ to $\partial \Omega$, so $f^{-1}(p)$ is well-defined.  Note also that if $f$ and $g$ are conformal maps from $\disk$ to $\Omega$, then $h \circ f$ has a non-tangential limit at $f^{-1}(p)$ if and only if $h \circ g$ has a non-tangential limit at $g^{-1}(p)$, since $g^{-1} \circ f$ is a disk automorphism.  This shows that 
 the notion of conformally non-tangential limit is well-defined.

 We have the following result of Beurling/Zygmund, phrased conformally invariantly.  
 \begin{theorem}[\cite{El-Fallah_etal_primer}]  \label{th:conf_inv_Beurling} Let $\Gamma$ be a strip-cutting Jordan curve in a 
 compact Riemann surface $R$ and let $\Omega$ be a component of the
 complement.  Assume that $\Omega$ is simply connected.
  Then for every $H \in \mathcal{D}_{\mathrm{harm}}(\Omega)$, $H$ has a conformally non-tangential limit at $p$ for all $p$ except possibly on a null set in $\Gamma$ with respect to $\Omega$.  If  $H_1,H_2 \in \mathcal{D}_{\mathrm{harm}}(\Omega)$ have the same conformally non-tangential boundary values except possibly on a null set in $\Gamma$, then $H_1= H_2$. 
 \end{theorem}
 This theorem follows immediately for the statement on the disk, together with conformal invariance
 of the definition.  
 We will shorten this statement to say that the boundary values exist ``conformally non-tangentially quasieverywhere'' or ``conformally non-tangentially'' with respect to $\Omega$.\\
 
From now on we will use the abbreviation``CNT'' for conformally non-tangentially.

 Let $\Gamma$ be an oriented Jordan curve in $\mathbb{C}$.  Let $\Omega_i$, $i=1,2$ denote the distinct connected components of the complement of $\Gamma$ in $\sphere$.  Temporarily let $\mathcal{H}(\Gamma,\Omega_i)$ denote the set of functions $h \in \mathcal{B}(\Gamma,\Omega_i)$ which are CNT limits of some $H \in \mathcal{D}(\Omega_i)$ quasi-everywhere.

 If $\mathcal{H}(\Gamma,\Omega_1)= \mathcal{H}(\Gamma,\Omega_2)$ then there would be induced maps of the
 Dirichlet spaces
 \[  r(\Omega_1,\Omega_2):\mathcal{D}_{\text{harm}}(\Omega_1) \rightarrow \mathcal{D}_{\text{harm}}(\Omega_2) \ \ \text{and} \ \ r(\Omega_2,\Omega_1):\mathcal{D}_{\text{harm}}(\Omega_2) \rightarrow \mathcal{D}_{\text{harm}}(\Omega_1). \]
If $\Gamma$ is a quasicircle, this is indeed the case, and in fact this characterizes quasicircles.
 \begin{theorem}[\cite{SchippersStaubach_jump}] \label{th:qc_both_sides_same_H}  Let $\Gamma$ be a Jordan curve in $\mathbb{C}$.    The following are equivalent.
 \begin{enumerate}
  \item $\Gamma$ is a quasicircle;
  \item $\mathcal{H}(\Gamma,\Omega_1) \subseteq \mathcal{H}(\Gamma,\Omega_2)$ and the map $r(\Omega_1,\Omega_2)$ is bounded
  with respect to the Dirichlet semi-norm;
  \item  $\mathcal{H}(\Gamma,\Omega_2) \subseteq \mathcal{H}(\Gamma,\Omega_1)$ and the map $r(\Omega_2,\Omega_1)$ is bounded
  with respect to the Dirichlet semi-norm.
 \end{enumerate}
 \end{theorem}
 Thus, for quasicircles in $\sphere$, we will denote $\mathcal{H}(\Gamma)=\mathcal{H}(\Gamma,\Omega_1)=\mathcal{H}(\Gamma,\Omega_2)$
 without ambiguity.
 
 In contrast to Theorem \ref{th:conf_inv_Beurling},  Theorem \ref{th:qc_both_sides_same_H} is non-trivial and not a tautological consequence of Beurling's theorem.  
 \begin{remark}  In \cite{SchippersStaubach_jump}, the boundary values are given as limits along images of radial lines under a conformal map $f:\disk \rightarrow \Omega$.  Since all elements of $\mathcal{D}_{\mathrm{harm}}(\Omega)$ have CNT limits quasieverywhere, and the limits along these curves must equal the conformally non-tangential limit when it exists, the theorem above is an immediate consequence.  
 \end{remark}

\end{subsection}
\begin{subsection}{Conformally non-tangential limits}

We now give the general definition of CNT limits.  We do this by replacing the map $f$ with the canonical collar chart.  However, it is not immediately obvious that the definition is independent of the choice of singularity in Green's function, so we temporarily make the dependence on the singularity explicit.  
\begin{definition}  Let $R$ be a compact surface and $\Gamma$ be a strip-cutting Jordan 
 curve $\Gamma$ which separates $R$ into two components, one of which is $\riem$. 
 Fix $q \in \riem$ and let $\phi$ be the canonical collar chart induced by Green's function with singularity 
 at $q$.  
 \begin{enumerate}
  \item
 For a function $H:\riem \rightarrow \mathbb{C}$, we say that $H$ has a CNT limit at $p \in \Gamma$ with respect to $(\riem,q)$ if $H \circ \phi^{-1}$ has a non-tangential limit at $\phi(p)$.  
 \item We say that $H$ has CNT boundary values quasieverywhere 
 with respect to $(\riem,q)$ if $H$ has CNT limits at all $p \in \Gamma$ except 
 possibly on a null set with respect to $\riem$.  
 \end{enumerate}
\end{definition}
We now show that the definition is independent of the singularity.
 \begin{theorem}  \label{th:CNT_point_independence}
  Let $\Gamma$ be a strip-cutting Jordan curve in a compact surface $R$ which separates $R$ into two components, one of which is $\riem$.  Fix points $q_1,q_2 \in \riem$.  Then a function $H:\riem \rightarrow \mathbb{C}$ has a $\mathrm{CNT}$ limit at $p \in \Gamma$ with respect to $(\riem,q_1)$ if and only if $H$ has a $\mathrm{CNT}$ limit with respect to $(\riem,q_2)$.    
 \end{theorem}
 \begin{proof}
  Let $\phi_1$ and $\phi_2$ be collar charts induced by $q_1$ and $q_2$ respectively.  Then 
  $\phi_2 \circ \phi_1^{-1}: A_1 \rightarrow A_2$ is a conformal map between some collar neighbourhoods $A_1$ and $A_2$ of 
  $\mathbb{S}^1$ in $\disk$.  By Lemma \ref{le:continuous_extension_collar} and Schwarz reflection \cite{Pommerenke_boundary_behaviour}, $\phi_2 \circ \phi_1^{-1}$ has a conformal extension to a map $g:A_1^* \rightarrow A_2^*$ where $A_1^*$ and $A_2^*$ are open neighbourhoods of $\mathbb{S}^1$ obtained by adjoining $\mathbb{S}^1$ and the symmetric regions under $z \mapsto 1/\bar{z}$.  
  
  Fix any Stolz angle $W$ at $p$, and denote by $B(p,\epsilon)$ the open disk of radius $\epsilon$ at $p$. Since $g$ is conformal at $\phi_2(p)$, for $\epsilon$ sufficiently small $g(B(p,\epsilon) \cap W)$ is contained in a possibly larger Stolz angle at $\phi_1(p)$, and a similar statement holds for $g^{-1}$.
  This completes the proof.
 \end{proof}
 Henceforth, we will say that $H$ has a CNT limit with respect to $\riem$.  
 
 \begin{remark}[conformal invariance of CNT limit]  \label{re:CNT_definition_confinv} It follows immediately from 
 conformal invariance of Green's function, that existence of CNT limits is conformally invariant in the following sense.  If $\riem' \subset R'$ and $\riem \subset R$ are bounded by strip-cutting Jordan curves, for any conformal map $f:\riem' \rightarrow \riem$ we have that $H:\riem \rightarrow \mathbb{C}$ has CNT boundary values at $p$ with respect to $(\riem,q)$ if and only if $H \circ f$
has CNT boundary values at $f^{-1}(p)$ with respect to $(\riem',f^{-1}(q))$.
By Theorem \ref{th:CNT_point_independence}, we therefore have that $H$ has a CNT limit with respect to $\riem$ at $p$ if and only if $H \circ f$ has a CNT limit at $f^{-1}(p)$ with respect to $\riem'$.  
\end{remark}

Finally, we observe that the proof of Theorem \ref{th:CNT_point_independence} does not depend on the fact that the collar charts are canonical.  Thus, we have the following.
\begin{theorem} \label{th:CNT_any_collar_chart}  Let $\Gamma$ be a strip-cutting Jordan curve in a compact surface $R$ which separates $R$ into two components, one of which is $\riem$. Let $H:\riem \rightarrow \mathbb{C}$.  
\begin{enumerate}
\item $H$ has a $\mathrm{CNT}$ limit at $p$ if and only if there is a collar chart $\phi$ in a collar neighbourhood of $\Gamma$ such that $H \circ \phi^{-1}$ has a non-tangential limit at $\phi(p)$.  
\item $H$ has $\mathrm{CNT}$ boundary values quasieverywhere on ${\Gamma}$ if and only if there is a collar chart $\phi$ in a collar neighbourhood of $\Gamma$ such that $H \circ \phi^{-1}$ has non-tangential limits quasieverywhere on $\mathbb{S}^1$.
\end{enumerate}
\end{theorem}

We see that this definition is independent of the regularity of the boundary by design.  In contrast, the main results, which involve limiting values from both sides of the curve, cannot be obtained from conformal invariance, and depend crucially on the regularity of the curve $\Gamma$.  The reader should keep this in mind through the rest of the paper.

\end{subsection}
\begin{subsection}{Boundary values of harmonic functions of bounded Dirichlet energy on doubly-connected domains}

In this section, we show that the existence of CNT boundary values of Dirichlet-bounded harmonic functions quasieverywhere only requires the harmonic function to exist on a 
collar neighbourhood.
\begin{theorem} \label{th:nearby_extension} Let $\Gamma$ be a Jordan curve in $\mathbb{C}$, and let $A$ be a collar neighbourhood of $\Gamma$.  Let $\riem$ be the component of the complement of $\Gamma$ in $\sphere$ which contains $A$. 
      If $H \in \mathcal{D}_\mathrm{harm}(A)$, then $H$ has $\mathrm{CNT}$ boundary values quasieverywhere on $\Gamma$ with respect to $\riem$.
\end{theorem}
\begin{proof}  Set $\Gamma_1 = \Gamma$ and let $\Gamma_2$ be the other Jordan curve bounding
 $A$.  Set $B_1=\riem$ and let $B_2$ be the component of the complement of $\Gamma_2$ containing $A$.
  Let $q' \in B_1 \backslash \text{cl} A$.
 Given any $H \in \mathcal{D}_\mathrm{harm}(A)$, by Corollary \ref{co:harm_in_out_double_decomp}
 we can write $H = h_1 + h_2 + c g_{q'}$ where $h_i \in \mathcal{D}_\mathrm{harm}(B_i)$, $g_{q'}$ is Green's function on $\riem=B_1$ and $c\in \mathbb{C}$.  Now $g_{q'}$ and $h_2$ are continuous on
 $\Gamma_1$, so in particular they have CNT boundary values quasieverywhere with respect to $\riem$.  Since $h_1 \in \mathcal{D}(B_1)$, by Theorem \ref{th:conf_inv_Beurling} it has CNT boundary values quasieverywhere
 in $\mathcal{H}(\Gamma,\riem)$.  Thus $H$ has CNT  boundary values quasieverywhere on $\Gamma$ with respect to $\riem$.   
\end{proof}

Let $\Gamma$ be a Jordan curve in $\mathbb{C}$ and let $\riem$ be one of the connected components of the complement. By Theorem \ref{th:conf_inv_Beurling} every element of $\mathcal{D}(\riem)$ has CNT boundary values quasieverywhere on $\Gamma$ with respect to $\riem$.  Define $\mathcal{H}(\Gamma,\riem)$ to be the set of functions in $\mathcal{B}(\Gamma,\riem)$ obtained in this way.  

\begin{theorem} \label{th:local_extension_enough} Let $\Gamma$ be a Jordan curve in $\mathbb{C}$. Let $A$ be a collar neighbourhood of $\Gamma$ and $\riem$ be the connected component of the complement of $\Gamma$ containing $A$. Then given any  $H \in \mathcal{D}(A)$, the boundary values in $h \in \mathcal{B}(\Gamma,\riem)$ 
are in $\mathcal{H}(\Gamma,\riem)$.
\end{theorem}
\begin{proof}  Let $B_1$, $B_2$, and the decomposition $H = h_1 + h_2 + c g_{q'}$ be as in
the proof of Theorem \ref{th:nearby_extension} above.   Since $h_2$ is harmonic on an open neighbourhood of $B_1^c$, it has bounded Dirichlet energy on the interior
 of the complement  
 of $B_1$.  Thus it has boundary values in $\mathcal{H}(\Gamma_1)$ by Theorem \ref{th:qc_both_sides_same_H}.  Since $h_1 \in \mathcal{D}_{\text{harm}}(B_1)$ and $g_{q'}$ is zero on $\Gamma$, this shows that the boundary values 
 of $H$ are in $\mathcal{H}(\Gamma)= \mathcal{H}(\Gamma,B)$.  
 
 If $\Gamma$ is not a quasicircle, let $f:\Omega \rightarrow \riem$ be a conformal map where $\partial \Omega$
 is a quasicircle and $\Omega \subset \mathbb{C}$ is simply connected.  The claim now follows 
 from conformal invariance (Remark \ref{re:CNT_definition_confinv}).  
\end{proof}

We then immediately have that
\begin{corollary}  \label{co:local_in_sphere}
 Let $\Gamma$ be a Jordan curve in $\sphere$. Let $\riem$ be one of the components of the complement of 
 $\Gamma$ in $\sphere$.    The following are equivalent.
 \begin{enumerate}
  \item $h \in \mathcal{H}(\Gamma,\riem)$;
  \item There is a Dirichlet bounded harmonic extension of $h$ to a collar nbhd $A$ in $\riem$;
  \item There is a Dirichlet bounded harmonic extension of $h$ to every collar nbhd $A$ in $\riem$;
 \end{enumerate}
 where by extension we mean that the boundary values equal $h$ $\mathrm{CNT}$ quasieverywhere.
\end{corollary}

For quasicircles, Theorem \ref{th:qc_both_sides_same_H} allows us to remove the condition that $A$ is in a fixed component.  
\begin{corollary}  \label{co:local_in_sphere_either_side}
 Let $\Gamma$ be a quasicircle in $\sphere$.  The following are equivalent.
 \begin{enumerate}
  \item $h \in \mathcal{H}(\Gamma)$;
  \item There is a Dirichlet bounded harmonic extension of $h$ to a collar nbhd $A$;
  \item There is a Dirichlet bounded harmonic extension of $h$ to every collar nbhd $A$;
 \end{enumerate}
 where by extension we mean that the boundary values equal $h$ $\mathrm{CNT}$ quasieverywhere.
\end{corollary}
\end{subsection}

\begin{subsection}{Boundary values of $\mathcal{D}_{\text{harm}}(\riem)$ for general $\riem$}
  \begin{theorem}  \label{th:analytic_annular_enough} Let $R$ be a compact Riemann surface and $\Gamma$ be an analytic Jordan curve in $R$ which separates $R$ into two components, one of which is $\riem$.  
  Let $A$ be a collar neighbourhood of $\Gamma$ in $\riem$.  For any $h \in \mathcal{D}_{\mathrm{harm}}(A)$, 
  $h$ has $\mathrm{CNT}$ boundary values quasieverywhere with respect to $\riem$.  Furthermore there is an $H \in \mathcal{D}_{\mathrm{harm}}(\riem)$ whose $\mathrm{CNT}$ boundary values agree with those of $h$ quasi-everywhere.  This $H$ is unique and given by the harmonic Sobolev extension of the Sobolev trace of $h$ in $H^{1/2}(\Gamma)$.
 \end{theorem}
 \begin{proof}
  We can shrink $A$ so that the inner boundary is also analytic.  Since $h \in \mathcal{D}_{\mathrm{harm}}(A)$, it has a Sobolev trace $\tilde{h} \in H^{1/2}(\Gamma)$.  
  Thus there is a unique $H \in \mathcal{D}_{\mathrm{harm}}(\riem)$ whose Sobolev trace equals $\tilde{h}$. 
  
  Let $B_1$ be the connected component of the complement of $\Gamma$ containing $A$, and let $B_2 = A \cup \text{cl} \riem^c$.  Fix $q \in \riem \backslash A$ and let $h= h_1 + h_2 + c g_q$ be the decomposition 
  guaranteed by Corollary \ref{co:harm_in_out_double_decomp}.  The restriction of $h_2$ to $\Gamma$ is a 
  $C^\infty$ function on a compact set, so it is in $H^{1/2}(\Gamma)$.  Therefore there is an $h_3 \in \mathcal{D}_{\text{harm}}(\riem)$ whose Sobolev trace equals $h_2$ almost everywhere on $\Gamma$.  But since $h_2$ is continuous on $\Gamma$, by the solution to the classical Dirichlet problem \cite{Ahlfors_Sario} there is a harmonic function which is continuous on the closure of $\riem$ which agrees with $h_2$ on $\Gamma$.  
  This function obviously equals $h_3$, so we see that $h_3$ extends continuously to $\Gamma$.  In particular, $h_3$ has CNT boundary values with respect to $\riem$ everywhere.  
 
 Since $h_2$ and $g_q$ are also continuous on $\Gamma$, they have CNT
 values everywhere on $\Gamma$ with respect to $\riem$.  Also, $h_1$ has CNT boundary values quasieverywhere on $\Gamma$ by Corollary \ref{co:local_in_sphere} since it lies in $\mathcal{D}_{\text{harm}}(\riem)$ (and similarly for $H$).  We conclude that $h$, $h_1 + h_2$, $h_1 + h_3$ all have CNT boundary values which agree quasieverywhere on $\Gamma$.  
  
  Thus the proof will be complete if we can show that $H = h_1 + h_3$.  
  But this follows immediately from the fact that the Sobolev trace of $h_3$ equals the Sobolev trace of $h_2$ almost everywhere, so the Sobolev trace of $H$ equals the Sobolev trace of $h_1 + h_2 = h$ almost everywhere.  
 \end{proof}
 
 Actually, a version of this theorem holds for strip-cutting Jordan curves, though we cannot identify 
 $H$ with a Sobolev extension.
 \begin{theorem}  \label{th:Jordan_annular_enough}
 Let $R$ be a compact Riemann surface and $\Gamma$ be a 
  strip-cutting Jordan curve in $R$ which separates $R$ into two components, one of which is $\riem$.  
  Let $A$ be a collar neighbourhood of $\Gamma$ in $\riem$.  For any $h \in \mathcal{D}_{\mathrm{harm}}(A)$, 
  $h$ has $\mathrm{CNT}$ boundary values quasieverywhere on $\Gamma$ with respect to $\riem$.  Furthermore, there is a unique $H \in \mathcal{D}_{\mathrm{harm}}(\riem)$ whose $\mathrm{CNT}$ boundary values agree with those of $h$ quasi-everywhere. 
 \end{theorem}
 \begin{proof}   There is a compact surface $R'$ and a $\riem'$ bounded by an analytic strip-cutting Jordan curve $\Gamma'$ in $R'$, and a conformal map $f:\riem' \rightarrow \riem$.  This can be obtained for example by completing $\riem$ to the double $R'$.  Let $\riem'$ be the subset of the double corresponding to $\riem$.   The inclusion map is conformal from $\riem$ to $\riem'$ and identifies $\partial \riem$ with an analytic curve $\Gamma'$.  
 
 We now apply Theorem \ref{th:analytic_annular_enough} to $\riem'$, $A' = f^{-1}(A)$, and $h' = h \circ f$, resulting in an $H' \in \mathcal{D}_{\text{harm}}(\riem')$ whose CNT limits agree with those of $h'$ 
 quasi-everywhere.
 Applying the conformal invariance of CNT limits (Remark \ref{re:CNT_definition_confinv}) completes the proof.
 \end{proof}
\begin{remark} The reader should observe that there is no Sobolev space $H^{1/2}(\Gamma)$ for general quasicircles $\Gamma$, and therefore there is no meaningful identification of CNT limits with Sobolev traces in this generality.
However, if the quasicircle is sufficiently regular, e.g. Ahlfors regular, then $H^{1/2}(\Gamma)$ can be replaced by an appropriate Besov space, see i.e. \cite{Johnsson} and \cite{Johnsson-Wallin}.  
 \end{remark} 
 
 We now make the following definition.
 \begin{definition}  Let $\Gamma$ be a strip-cutting Jordan curve in a compact Riemann surface $R$, which separates $R$ into two components, one of which is $\riem$.  Let $\mathcal{H}(\Gamma,\riem)$ be the set of functions in 
 $\mathcal{B}(\Gamma,\riem)$ which are the boundary values quasi-everywhere of an element of $\mathcal{D}(\riem)_{\text{harm}}$.  
 \end{definition}
 
 By Theorems \ref{th:CNT_any_collar_chart} and \ref{th:Jordan_annular_enough}, we immediately have
 the following characterization of $\mathcal{H}(\Gamma,\riem)$.
 \begin{corollary}  \label{co:Jordan_annular_all_same}
  Let $R$ be a compact Riemann surface and $\Gamma$ be a strip-cutting Jordan curve which separates $R$ into two components, one of which is $\riem$.  Then $u \in \mathcal{H}(\Gamma,\riem)$ if and only if there is a collar neighbourhood $A$ and an $h \in \mathcal{D}_{\mathrm{harm}}(A)$ whose $\mathrm{CNT}$ boundary values equal $u$ quasi-everywhere.
 \end{corollary}
 
 Later, we will see that for quasicircles, $\mathcal{H}(\Gamma,\riem)$ is independent of the choice of component of the complement of $\Gamma$.

 By definition, CNT limits in a collar neighbourhood of $\mathbb{S}^1$ are non-tangential limits.  Thus 
 by Theorem \ref{th:Jordan_annular_enough} we immediately have the following.   
 \begin{corollary} \label{co:annulus_nontangential}  Let $\mathbb{A} = \{ z: r<|z|<1 \}$.  If 
  $h \in \mathcal{D}_\mathrm{harm}(\mathbb{A})$ then $h$ has non-tangential limits except possibly 
  on a Borel set of capacity zero in $\mathbb{S}^1$. 
 \end{corollary}

   The map from $\mathcal{D}_{\text{harm}}(A)$ to $\mathcal{D}_{\text{harm}}(\riem)$ obtained by taking boundary values and then extending to $\riem$ is bounded.  
   This fact will be useful in future applications so we record it here.  
   Let $\Gamma$ be a strip-cutting Jordan curve in $R$ which separates $R$.  Let $A$ be a collar neighbourhood of $\Gamma$ contained in one of the connected components $\riem_1$.  
  By  Theorem \ref{th:Jordan_annular_enough}  the CNT 
  boundary values of $\mathcal{D}(A)$ exist. We then define 
  \begin{align*}
   \mathfrak{G}: \mathcal{D}_{\text{harm}}(A) & \rightarrow \mathcal{D}_{\text{harm}}(\riem_1)  \\
   h & \mapsto \tilde{h}
  \end{align*}
  where $\tilde{h}$ is the unique element of $\mathcal{D}_{\text{harm}}(\riem_1)$ with boundary
  values equal to $h$.  
 
  \begin{theorem} \label{th:iota_bounded} Let $\Gamma$ be a strip-cutting Jordan curve which separates $R$, and let $\riem$ be one of the components of the complement.  Let $A$ be a 
   collar neighbourhood of $\Gamma$  in  $\riem$.  Then the associated map $\mathfrak{G}:\mathcal{D}_{\mathrm{harm}}(A) \rightarrow \mathcal{D}_{\mathrm{harm}}(\riem)$ is bounded.  
  \end{theorem}
  \begin{proof} 
  The boundedness estimates that we are aiming to prove is $\Vert \mathfrak{G}(A,\Omega)h \Vert_{\mathcal{D}_{\mathrm{harm}}(\riem)}\leq C \Vert h\Vert_{\mathcal{D}_{\mathrm{harm}}(A)} $, which is valid for $h\in \mathcal{D}_{\mathrm{harm}}(A).$ First, observe that we can assume that the inner boundary of $A$ is analytic.  To see this, let $A' \subseteq A$ be collar neighbourhood whose inner boundary is analytic.  Since $\| \left. h \right|_{A'} \|_{\mathcal{D}_{\text{harm}}(A')} \leq \| h\|_{\mathcal{D}_{\text{harm}}(A)}$, it is enough to show that $\mathfrak{G}$ is bounded 
   with respect to the $\mathcal{D}_{\text{harm}}(A')$ norm.  
  
  Next, observe that because CNT boundary values and the Dirichlet norms are conformally invariant, it is enough to prove this for an analytic strip-cutting curve $\Gamma$, and
  this can be arranged for example by embedding $\riem$ in its double.  Thus, we can
  assume that $\partial A'$ is analytic.   
 Furthermore, the existence of the solution to the Dirichlet problem with boundary data in Sobolev spaces (see e.g. Proposition 4.5 on page 334 in \cite{Taylor})  and the fact that $\Gamma \subsetneq \partial A$, yields that
  \[  \| \left. h \right|_{\Gamma} \|_{H^{1/2}(\Gamma)} \leq  \| \left. h \right|_{\partial \Omega} \|_{H^{1/2}(\partial A)}  \leq C_1 \| h \|_{H^1(A)}. \]
  
  Also, the harmonic Sobolev extension $H$ of $\left. h \right|_{\Gamma}$ satisfies 
  \[  \| H \|_{{H}^1(\riem)} \leq C_2 \| \left. h \right|_{\Gamma} \|_{H^{1/2}(\Gamma)}  \]
  (see e.g. Proposition 1.7 on page 360 in \cite{Taylor}). This together with the estimate for $\Vert h|_\Gamma \Vert_{H^{1/2}(\Gamma)}$ above yields that 
  \begin{equation}\label{Sobolev estimate for bounce}
  \Vert  \mathfrak{G}(A,\riem)h \Vert _{H^1(\riem)} \leq  \Vert h\Vert _{H^1(A)}.
  \end{equation}
  Now if
  one applies \eqref{Sobolev estimate for bounce} to the harmonic function $h-h_{A}$ where $h_{A}$ is the average of $h$ given by $\frac{1}{|A|} \int_{A} h$, then one has that
$$\Vert \mathfrak{G}(A,\Omega)h-\mathfrak{G}(A,\Omega)h_{A}\Vert_{H^1(\riem)} \leq C \Vert h-h_{A}\Vert_{H^1(A)}. $$
Moreover we know that
$$\Vert \mathfrak{G}(A,\Omega)h \Vert_{\mathcal{D}_{\mathrm{harm}}(\riem)}= \Vert \mathfrak{G}(A,\Omega)h-\mathfrak{G}(A,\Omega)h_{A}\Vert_{\mathcal{D}_{\mathrm{harm}}(\riem)} \leq \Vert \mathfrak{G}(A,\Omega)h-\mathfrak{G}(A,\Omega)h_{A}\Vert_{H^1(\riem)}$$ and that $$\Vert h-h_{A}\Vert_{H^1(A)}= \Vert h-h_{A}\Vert_{\mathcal{D}_{\mathrm{harm}}(A)}+ \Vert h-h_{A}\Vert_{L^2(A)}\leq \Vert h\Vert_{\mathcal{D}_{\mathrm{harm}}(A)}+ \Vert h\Vert_{\mathcal{D}_{\mathrm{harm}}(A)}= 2\Vert h\Vert_{\mathcal{D}_{\mathrm{harm}}(A)},$$ where the inequality $ \Vert h-h_{A}\Vert_{L^2(A)}\leq C \Vert h\Vert_{\mathcal{D}_{\mathrm{harm}}(A)}$ is nothing but the Poincar\'e-Wirtinger inequality. This yields the estimate $\Vert \mathfrak{G}(A,\Omega)h \Vert_{\mathcal{D}_{\mathrm{harm}}(\riem)}\leq 2C \Vert h\Vert_{\mathcal{D}_{\mathrm{harm}}(A)} ),$ and completes the proof.
  \end{proof}

\end{subsection}
\begin{subsection}{Transmission in quasicircles}

 In this section we show that if $\Gamma$ is a quasicircle in a compact surface $R$, and $\riem_1$ and $\riem_2$ are the connected components of $R \backslash \Gamma$, then  
 the CNT boundary values quasieverywhere of any $h_1 \in \mathcal{D}_{\text{harm}}(\riem_1)$ are also CNT boundary values of a unique $h_2 \in \mathcal{D}_{\text{harm}}(\riem_2)$ quasieverywhere.  This ``transmission'' of $h_2$ is bounded with respect to the Dirichlet semi-norm.  
 
 We first show the result for analytic Jordan curves $\Gamma$.
 
 \begin{theorem}  \label{th:trans_exists_analytic} Let $R$ be a compact Riemann surface and $\Gamma$ be an analytic Jordan curve separating 
  $R$ into components $\riem_1$ and $\riem_2$.  Then $\mathcal{H}(\Gamma,\riem_1)= \mathcal{H}(\Gamma,\riem_2)$.  That is, for every $h_1 \in \mathcal{D}_{\mathrm{harm}}(\riem_1)$, there is an $h_2 \in \mathcal{D}_{\mathrm{harm}}(\riem_2)$ so that the $\mathrm{CNT}$ boundary values of $h_1$ and $h_2$ agree quasieverywhere on $\Gamma$.
  Furthermore, $h_2$ is unique and given by the harmonic Sobolev extension of the trace of $h_1$ in $H^{1/2}(\Gamma)$. 
 \end{theorem}
 \begin{proof}  First observe that $\Gamma$ is in particular a quasicircle, so a set is null with respect to $\riem_1$ if and only if it is null with respect to $\riem_2$ by Theorem \ref{th:null_sets_both_sides}.  
 
 Since $\Gamma$ is an analytic curve, by definition there is a doubly connected neighbourhood $U$ of 
 $\Gamma$ and a chart $\phi:U \rightarrow \mathbb{A}$ such that $\phi(\Gamma) = \mathbb{S}^1$.  We may choose 
 $\mathbb{A}$ so that it is an annulus $\mathbb{A} = \{ z : r <|z|<1/r \}$ for some $r<1$, and let $A_1 = U \cap \riem_1$ and $A_2 = U \cap \riem_2$.    
 
 Let $h_1 \in \mathcal{D}_{\text{harm}}(\riem_1)$.  Now $h_1$ has CNT boundary values quasieverywhere, and therefore $\left. h_1 \circ \phi \right|_{A_1}$ has non-tangential boundary values on $\mathbb{S}^1$ by definition, and also $h_1$ is in  $\mathcal{D}_{\mathrm{harm}}(r<|z|<1)$.  Hence $u(z)=h_1 \circ \phi (1/\bar{z})$ is a harmonic function on $\mathcal{D}_{\mathrm{harm}}(1<|z|<1/r)$.  Now since $z \mapsto 1/\bar{z}$ takes wedges inside $\disk$ to wedges inside $\text{cl} (\disk^{c})$, one readily sees that $u$ has non-tangential boundary values agreeing with those of $h_1 \circ \phi$, everywhere the latter boundary values exist.  Thus $u \circ \phi^{-1} \in \mathcal{D}_{\mathrm{harm}}(A_2)$ and $u \circ \phi^{-1}$ has CNT boundary values agreeing with those of $h_1$.  
 
 The claim now follows by applying Theorem \ref{th:analytic_annular_enough} to $u \circ \phi^{-1}$.  
\end{proof}
 
 For quasicircles, Theorem \ref{th:trans_exists_analytic} motivates the following definition for analytic Jordan curves 
 separating $R$ into $\riem_1$ and $\riem_2$.
 \begin{definition}
We define the transmission operator 
 \[  \mathfrak{O}(\riem_1,\riem_2):\mathcal{D}_\mathrm{harm}(\riem_1) \rightarrow \mathcal{D}_\mathrm{harm}(\riem_2)  \]
through an analytic Jordan curve $\Gamma$, to be the function which takes $h \in \mathcal{D}_\mathrm{harm}(\riem_1)$ to the harmonic function on $\mathcal{D}_\mathrm{harm}(\riem_2)$ with the same conformally radial boundary values.  Define
 \[  \mathfrak{O}(\riem_2,\riem_1):\mathcal{D}_\mathrm{harm}(\riem_2) \rightarrow \mathcal{D}_\mathrm{harm}(\riem_1) \]
 similarly.
\end{definition}
 Below, we show that this is well-defined for quasicircles.

 Next we investigate the boundedness of the transmission operator in various settings (related to the regularity of $\Gamma$).
 \begin{theorem}\label{th:analytic_transmission_bounded}
  Let $\Gamma$ be an analytic strip-cutting Jordan curve in a compact Riemann surface $R$.  Assume that $\Gamma$ 
  separates $R$ into connected components $\riem_1$ and $\riem_2$.  Then 
  $\mathfrak{O}(\riem_1,\riem_2):\mathcal{D}_{\mathrm{harm}} (\riem_1)\rightarrow \mathcal{D}_{\mathrm{harm}} (\riem_2)$ is bounded.  
 \end{theorem}
 \begin{proof} 
  Note that in contrast to the proof of Theorem \ref{Thm:main boundedness of transmission}, here the assumption that the curve is analytic cannot be removed by conformal invariance, since a conformal map of one side of $\Gamma$ does not preserve CNT boundary values taken from the other.   
 
  The transmission procedure is as follows.  
 Start with $h_1\in \mathcal{D}_\mathrm{harm}(\Sigma_1)$ which by Theorem \ref{equiv and sobolevs in riem} is also in $H^{1}(\Sigma_1)$, so we can restrict it to the boundary to obtain a trace which is in $H^{1/2}(\Gamma)$. 
 Furthermore  
\begin{equation}\label{trace estim 3}
\Vert h_1|_{\Gamma} \Vert_{H^{1/2}(\Gamma)}\leq C\Vert h_1\Vert _{H^ 1(\riem_1)}.
\end{equation}
Now solve the Dirichlet problem with $h_1|_\Gamma$ as Dirichlet data to obtain a unique harmonic function $h_2:= \mathfrak{O}(\riem_1,\riem_2) h_1$ in $H^1 (\Sigma_2)$ that verifies 
\begin{equation}\label{elliptic estim 3}
\Vert h_2\Vert_{H^1(\riem_2)} \leq C\Vert h_1|_\Gamma \Vert_{H^{1/2}(\Gamma)},
\end{equation} 
see e.g. Proposition 1.7 on page 360 in \cite{Taylor}.

Finally combining the estimates \eqref{trace estim 3} and \eqref{elliptic estim 3} yields that

\begin{equation}\label{trace is bounded}
\Vert \mathfrak{O}(\riem_1,\riem_2) h_1\Vert_{H^1(\riem_2)}\leq C  \Vert h_1\Vert_{H^1(\riem_1)}.
\end{equation} 

Now for a compact Riemannian manifold $M$, with or without boundary, one has for functions $f\in H^1(M)$ with $\int_{M} f=0$ the Poincar\'e inequality $\int_{M} |f|^2 \leq C \int_{M} |\nabla f|^2$, see e.g. \cite{Li} page 75 estimates (1.6) and (1.7). Therefore if $\int_{M} f\neq 0$ then one can apply the aforementioned Poincar\'e inequality to $g:= f-f_M$ with $f_{M}=\frac{1}{|M|}\int_{M} f$ for which one has $\int_M g=0.$ This yields that $\int_{M} |f-f_M|^2 \leq C \int_{M} |\nabla f|^2$ for $f\in H^1(M)$. Now   applying estimate \eqref{trace is bounded} to the harmonic function $h_{1}- h_{1\Sigma_{1}}$ and follow the same reasoning as in the proof of Theorem \ref{th:iota_bounded}, one readily sees that $\Vert \mathfrak{O}(\riem_1,\riem_2) h_1\Vert_{\mathcal{D}_\mathrm{harm}(\riem_2)}\leq C  \Vert h_1\Vert_{\mathcal{D}_\mathrm{harm}(\riem_1)}.$ 
\end{proof}
 
 To extend this result to quasicircles, we require the following Lemma.

 \begin{lemma} \label{le:quasicircle_straighten}  Let $R$ be a compact surface and 
  $\Gamma$ be a quasicircle, separating $R$ into components $\riem_1$ and $\riem_2$.  Let $\Omega$ be a collar neighbourhood of $\Gamma$ in $\riem_2$. There is a Riemann surface $S$ and a quasiconformal map
  $f:R \rightarrow S$ such that $f$ is conformal on the complement of the closure of $\Omega$ and $f(\Gamma)$ is an analytic Jordan curve.
 \end{lemma}
 \begin{proof}
 Let $A$ be a doubly-connected neighbourhood of 
  $\Gamma$ such that the closure of $A$ is contained in $\Omega$.  Let $\Gamma_1$ and $\Gamma_2$ be the boundaries of $A$ in $\riem_1$ and 
  $\riem_2$ respectively; these can be taken as regular as desired (say, analytic Jordan curves).
  Denote $\mathbb{A}_{r,s} = \{ z\,:\, r<|z|<s \}$.  
  Let $\psi:A \rightarrow \mathbb{A}_{r,1/r}$ be a biholomorphism (which exists for some value of $r \in (0,1)$  by a classical canonical mapping theorem \cite{Neharibook}).  If $\Gamma_1$ and $\Gamma_2$ 
  are analytic, then there is a biholomorphic extension of $\psi$ to an open neighbourhood of the closure of $A$ 
  taking $\Gamma_i$ to the boundary curves; for definiteness, we assume that $\Gamma_1$ maps to $|z|=r$ and $\Gamma_2$ maps to $|z|=1/r$ (which can of course be arranged by composing $\psi$ by $z \mapsto 1/z$).  
  Let $E_1$ denote the region bounded by $\psi(\Gamma_1)$ and $\psi(\Gamma)$, and $E_2$ denote the region 
  bounded by $\psi(\Gamma_2)$ and $\psi(\Gamma)$.  
  
  We claim that there is a map $\phi:\mathbb{A}_{r,1/r}$ to an annulus $\mathbb{A}_{s,1/r}$ for some 
  $0<s < 1 <1/r$ with the following properties:
  \begin{enumerate}
   \item $\phi$ is quasiconformal
   \item $\phi$ takes $\psi(\Gamma)$ onto $\mathbb{S}^1$;
   \item $\phi$ is a biholomorphism from $E_1$ onto $\mathbb{A}_{s,1}$;
   \item $\phi$ has an analytic extension to an open neighbourhood of $|z|=r$;
   \item $\phi$ is the identity on $|z|=1/r$.  
  \end{enumerate}
  To see this, again applying the classical canonical mapping theorem, there is a biholomorphic map $G_1:E_1 \rightarrow \mathbb{A}_{s,1}$ for some $s \in (0,1)$.  This has a homeomorphic extension from $\Gamma$ 
  to $\mathbb{S}^1$ by Lemma \ref{le:continuous_extension_collar}.  	
  By \cite[Corollary 4.1]{Radnell_Schippers_monster}, there is a quasiconformal map $G_2:E_2 \rightarrow 
  \mathbb{A}_{1,1/r}$ which equals $G_1$ on $\psi(\Gamma)$ and the identity on $|z|=1/r$.  
 Set
  \[  \phi(z) = \left\{ \begin{array}{cc}  G_1(z) & z \in E_1 \cup \Gamma \\ G_2(z) & z \in E_2.    \end{array} \right. \]
   Since $\phi$ is a homeomorphism and $\Gamma$ has measure zero, $\phi$ is quasiconformal.
  All the properties then follow immediately, except (4), which is a consequence of the Schwarz reflection principle.
  
  We now construct the surface $S$ by sewing as follows.  Remove the closure of $A$ from $R$ to obtain 
  two disjoint Riemann surfaces $R_1$ and $R_2$.  Join $\mathbb{A}_{s,1/r}$ to $R_1$ and $R_2$, using 
  $\left.\psi \circ \phi \right|_{\Gamma_1}$ to identify points on $\Gamma_1$ with their image on $|z|=s$, 
  and $\left. \psi \circ \phi \right|_{\Gamma_2}$ to identify points on $\Gamma_2$ with 
  their image on $|z|=1/r$.  The resulting $S$ is a Riemann surface with a unique complex structure 
  compatible with that on $R_1$, $\mathbb{A}_{s,1/r}$, and $R_2$ by  \cite[Theorem 3.3]{Radnell_Schippers_monster}.  (For analytic curves and parametrizations this is standard, see e.g. \cite{Ahlfors_Sario}; however to sew $\Gamma_2$ one needs the stronger result).  Observe that 
  $\psi \circ \phi (\Gamma)$ is an analytic curve in $S$, which can be seen by using the identity map as a chart.
  
  The map from $R$ to $S$ given by 
  \[  f(z) =  \left\{  \begin{array}{cc}  z & z \in R_1 \cup \Gamma_1 \\  
    \psi \circ \phi(z) & z \in A \\ z & z \in R_2 \cup \Gamma_2.  
  \end{array} \right.\]
  is easily checked to be well-defined by applying the equivalence relation on the seams.  It is quasiconformal on each patch, with homeomorphic extensions to $\Gamma_1$ and $\Gamma_2$ and since the seams and their images are quasicircles and thus have measure zero, $f$ is quasiconformal on $R$.  The map $f$ is clearly conformal on $R_2$, which includes $\riem_2$ minus the closure of $\Omega$. Since the identification at $\Gamma_1$ is analytic, and $\psi \circ \phi$ is conformal on $E_1$, $f$ is conformal on $\riem_1$ \cite{Ahlfors_Sario}. This proves the claim.  
 \end{proof}
  \begin{theorem} \label{th:transmission_exists_quasicircle}  Let $R$ be a compact Riemann surface, and $\Gamma$ be a quasicircle in $R$.
  Assume that $\Gamma$ separates $R$ into two components $\riem_1$ and $\riem_2$.  
  
  The following are equivalent.
  \begin{enumerate}
   \item $h \in \mathcal{H}(\Gamma,\riem_1)$;
   \item $h \in \mathcal{H}(\Gamma,\riem_2)$;
   \item $h$ is the $\mathrm{CNT}$ boundary values quasieverywhere of an element of 
   $\mathcal{D}(A)$ for some collar neighbourhood $A$ of $\Gamma$ in $\riem_1$;
   \item $h$ is the $\mathrm{CNT}$ boundary values quasieverywhere of an 
    element of $\mathcal{D}(A)$ for some collar neighbourhood $A$ of $\Gamma$ in $\riem_2$.  
  \end{enumerate}
 \end{theorem}
 \begin{proof}
  By Theorem \ref{co:Jordan_annular_all_same} (1) is equivalent to (3) and (2) is equivalent to (4).  
  
  Assume that (1) holds.  Let $\psi:A \rightarrow \mathbb{A}$ be a doubly-connected chart in a neighbourhood $A$ of $\Gamma$.  Let $h \in \mathcal{H}(\Gamma,\riem_1)$, so that it has 
  an extension $H \in \mathcal{D}_\mathrm{harm}(\riem_1)$.  Then $H \circ \psi^{-1}$ is in 
  $\mathcal{D}(\psi(A \cap \riem_1))$.  Thus its boundary values $h \circ \psi^{-1}$ 
  are in $\mathcal{H}(\psi(\Gamma))$ by Corollary \ref{co:local_in_sphere}.  Let $\Omega_1$ and $\Omega_2$ be the components of the 
  complement of $\psi(\Gamma)$ in $\sphere$, where $\Omega_1$ contains $\psi(A \cap \riem_1)$.  Since $\psi(\Gamma)$ is a quasicircle, the CNT boundary values of $h \circ \psi^{-1}$ are in both $\mathcal{H}(\Gamma,\Omega_1)$ and $\mathcal{H}(\Gamma,\Omega_2)$ by Theorem \ref{co:local_in_sphere_either_side}.  Thus there is a $G \in \mathcal{D}_\mathrm{harm}(\Omega_2)$ with CNT boundary values equal to $h \circ \psi^{-1}$
  quasieverywhere.  
  Now $G \circ \psi \in \mathcal{D}(A \cap \riem_2)$,  so by Theorem \ref{th:Jordan_annular_enough} it
  has CNT boundary values quasieverywhere in $\mathcal{H}(\Gamma,\riem_2)$; by conformal invariance 
  of CNT boundary values (Remark \ref{re:CNT_definition_confinv}) these equal $h$ quasieverywhere. Similarly (2) implies (1).
 \end{proof}
 
 So for quasicircles, we can now define
 \[  \mathcal{H}(\Gamma) = \mathcal{H}(\Gamma,\riem_1)=\mathcal{H}(\Gamma,\riem_2).   \]

 Furthermore, this shows that the transmission operator $\mathfrak{O}(\riem_1,\riem_2)$ is defined for general quasicircles.  
In order to prove the boundedness of the transmission operator when the boundary curve is a quasicircle, we will need the following lemma:
 \begin{lemma} \label{le:annulus_nontangential agrees with sobolev}  Let $\mathbb{A} = \{ z: r<|z|<1 \}$.  If 
  $h \in H^1(\mathbb{A})$ then the non-tangential limits agree almost everywhere with the boundary trace taken in the sense of Sobolev spaces. 
 \end{lemma}
 \begin{proof}
 Since the annulus is an example of an $(\varepsilon,\delta)$ domain in the sense of Definition 1.1 of \cite{brewster_mitrea} (the boundaries are two smooth bounded curves), using Theorem 8.7 (iii) in \cite{brewster_mitrea} with $D=\partial\mathbb{A}=\{|z|=r\}\cup \{|z|=1\}$, which is $1$-Ahlfors regular, and taking $s=1,$ $p=2$ and $n=2$, we have that  their condition $s-\frac{n-d}{p}= 1-\frac{2-1}{2}=\frac{1}{2}\in (0,\infty)$ is satisfied. Thus, Theorem 8.7 (iii) in \cite{brewster_mitrea} yields that the Sobolev trace belonging to $H^{1/2}(D)$ agrees almost everywhere (since the $1$-dimensional Hausdorff measure on $D$ is the 1-dimensional Lebesgue measure) with the non-tangential limit of the function $h\in H^1 (\mathbb{A}).$ 
 \end{proof}

 Now we are ready to state and prove the main result concerning the Dirichlet space boundedness of the transmission operator.
 
 \begin{theorem}\label{Thm:main boundedness of transmission} Let $R$ be a compact Riemann surface, and let $\Gamma$ be a quasicircle 
  which separates $R$ into two connected components $\riem_1$ and $\riem_2$.  Then 
  $\mathfrak{O}(\riem_1,\riem_2)$ and $\mathfrak{O}(\riem_2,\riem_1)$ are bounded with respect to the Dirichlet 
  semi-norm.  
 \end{theorem} 
 \begin{proof}  The proof proceeds as follows.  We need only show that $\mathfrak{O}(\riem_1,\riem_2)$ is bounded, 
 since we can simply switch the roles of $\riem_1$ and $\riem_2$.  
By Theorem \ref{th:analytic_transmission_bounded} the theorem holds when $\Gamma$ is an analytic curve. We apply Lemma \ref{le:quasicircle_straighten} to map a general quasicircle to an analytic one.  By using the Dirichlet principle, we can show that the original transmission is bounded.  
  Throughout the proof, all of the transmissions take place through quasicircles (at worst), so the transmissions exist by Theorem \ref{th:transmission_exists_quasicircle}  and the boundary values of the transmitted function exist CNT quasieverywhere and agree with the original function CNT quasieverywhere.

Now assume that $\Gamma$ is a quasicircle, and not necessarily an analytic curve. Let $f:R \rightarrow R'$ be the quasiconformal map obtained from Lemma \ref{le:quasicircle_straighten}. Denote $f(\riem_i)=\riem_i'$ for $i=1,2$ and $f(\Gamma)$ by $\Gamma'$.

  Fix $h \in \mathcal{D}_{\text{harm}}(\riem_1)$.  Now $h \circ f^{-1} \in \mathcal{D}_{\text{harm}}(\riem_1')$ and thus
  it has a transmission $\mathfrak{O}(\riem_1',\riem_2') (h \circ f^{-1})$, and furthermore 
  \begin{equation} \label{eq:temp1_reflection}
    \| \mathfrak{O}(\riem_1',\riem_2') (h \circ f^{-1}) \|_{\mathcal{D}(\riem_2')} 
       \leq C \| h \circ f^{-1} \|_{\mathcal{D}(\riem_1')} = C \| h \|_{\mathcal{D}(\riem_1)}      
  \end{equation}
  by Theorem \ref{th:analytic_transmission_bounded}.  
  Since $f$ is a conformal map from $\riem_1$ to $\riem'_1$,  Remark \ref{re:CNT_definition_confinv} 
  applies.
  
 Now let $g:\riem_2'' \rightarrow \riem_2$ be a conformal map where 
  $\riem_2''$ is an analytically bounded Riemann surface.  Observe that $g$ is independent of $f$; we will need this ahead. Such a map exists by Lemma \ref{le:quasicircle_straighten} (or more directly by embedding $\riem_2$ in its double).  The function 
  $[\mathfrak{O}(\riem_1',\riem_2') (h \circ f^{-1})] \circ f \circ g$ is in 
  $\dot{H}^1(\riem_2'')$ by change of variables and the fact that
  $[\mathfrak{O}(\riem_1',\riem_2') (h \circ f^{-1})] \circ f$ is in $\dot{H}^1(\riem_2)$ by quasi-invariance of the Dirichlet norm. By Theorem \ref{equiv and sobolevs in riem},   $[\mathfrak{O}(\riem_1',\riem_2') (h \circ f^{-1})] \circ f \circ g$ is in 
  $H^1(\riem_2'')$. \\
  
  Now let $\tilde{u}$ be the Sobolev trace of $[\mathfrak{O}(\riem_1',\riem_2') (h \circ f^{-1})] \circ f \circ g$
  in $H^{1/2}(\Gamma'')$, which exists because $\riem_2''$ is analytically bounded.  We claim that  
  \begin{equation} \label{eq:step_one_final_proof}
        \tilde{u} =  h \circ g {\text{ a.e.}} 
  \end{equation}
  Roughly, this can be seen by cancelling $f^{-1} \circ f$ in the expression $[\mathfrak{O}(\riem_1',\riem_2') (h \circ f^{-1})] \circ f \circ g$, but the proof requires careful comparison of Sobolev and CNT boundary values of the transmission.  
 
  Let $\phi$ be a collar chart on a doubly connected neighbourhood $V$ of $\Gamma''$ in 
  the double of $\riem_2''$.  Choose the domain of definition $V$ of $\phi$ to be such that it contains an analytic curve $\gamma$ in $\riem_2$ which is such that $f$ is conformal on an open neighbourhood of $g(\gamma)$.  This can be accomplished by adjusting $f$ if necessary, which is possible because $g$ was independent of $f$.     
  
Now let $\psi$ be a collar
  chart on a collar neighbourhood $U$ of $\Gamma'$ in $\riem_2$ which contains $f \circ g(V \cap \riem_2'')$.  In particular, $\psi \circ f \circ g(\gamma)$ is an analytic curve.  It can be arranged that $\psi$ takes $\Gamma'$ to $\mathbb{S}^1$ by composing with a conformal map.  We need to show that the map  $\psi \circ f \circ g \circ \phi^{-1}$ extends to a quasisymmetry of $\mathbb{S}^1$.   To see this, restrict this map to 
 the region $W$ bounded by $\mathbb{S}^1$ and $\phi(\gamma)$.  Since $\phi(\gamma)$ and its image $\alpha$ is an analytic curve 
 (and in particular a quasicircle), and $\psi \circ f\circ g \circ \phi^{-1}$ is a real analytic diffeomorphism from $\phi(\gamma)$ to its image, there is a quasiconformal map $\Phi$ taking the region $D$ bounded by $\phi(\gamma)$ to the region $D'$ bounded by $\alpha$, with boundary values 
 $\psi \circ f \circ g \circ \phi^{-1}$.  The map $\tilde{\Phi}:\disk \rightarrow \disk$ given by 
 sewing these two functions together along $\phi(\gamma)$, that is
 \[   \tilde{\Phi}(z) = \left\{ \begin{array}{cc} \psi \circ f\circ g \circ \phi^{-1}(z) &  z \in W \cup \phi(\gamma)  \\ \Phi(z) & z \in W, \end{array}
 \right.   \]
  is a homeomorphism on $\disk$.  Since it is quasiconformal except on the measure zero set $\phi(\gamma)$, it is quasiconformal.  Thus by the Ahlfors-Beurling extension theorem, $\psi \circ f \circ g \circ \phi^{-1}$ extends to a quasisymmetry of $\mathbb{S}^1$ as claimed.
  
  Now
  $\mathfrak{O}(\riem_1',\riem_2')[h \circ f^{-1}] \circ f \circ g \circ \phi^{-1}$ is in the Dirichlet space of an annulus 
  $\mathbb{A}$.  Thus by Corollary \ref{co:annulus_nontangential} it has non-tangential boundary values
  except possibly on a null set in $\mathbb{S}^1$.  Since $\psi \circ f \circ g \circ \phi^{-1}$ (and its inverse) is a quasiconformal map of $\disk^+$, it and its inverse take Stolz angles inside larger Stolz angles by a result of P. Jones and L. Ward \cite{Jonesward} (see section 3 proof of Theorem 1.1 in that paper). Therefore since the harmonic function
  \begin{equation}  
  [ \mathfrak{O}(\riem_1',\riem_2')[h \circ f^{-1}] \circ f \circ g \circ \phi^{-1}] \circ (\psi \circ f \circ g \circ \phi^{-1} )^{-1} = [ \mathfrak{O}(\riem_1',\riem_2')[h \circ f^{-1}] \circ \psi^{-1} 
 \end{equation}
  has non-tangential limits quasi-everywhere by Lemma \ref{co:annulus_nontangential}, so does $[\mathfrak{O}(\riem_1',\riem_2') (h \circ f^{-1})] \circ f \circ g \circ \phi^{-1}$. Furthermore, the non-tangential limits of $\mathfrak{O}(\riem_1',\riem_2')[h \circ f^{-1}] \circ \psi^{-1}$ equal $h \circ f^{-1} \circ \psi^{-1}$, so since $\psi \circ f \circ g \circ \phi^{-1}$ has a homeomorphic extension to $\mathbb{S}^1$ which is a quasisymmetry, the non-tangential limits of $[\mathfrak{O}(\riem_1',\riem_2') (h \circ f^{-1})] \circ f \circ g \circ \phi^{-1}$ equal $h \circ g \circ \phi^{-1}$. 
  
  Now Lemma \ref{le:annulus_nontangential agrees with sobolev} shows that the non-tangential limits of $[\mathfrak{O}(\riem_1',\riem_2') (h \circ f^{-1})] \circ f \circ g \circ \phi^{-1}$ equal the Sobolev trace.  Thus the Sobolev trace in $H^{1/2}(\mathbb{S}^1)$ of 
  $[\mathfrak{O}(\riem_1',\riem_2') (h \circ f^{-1})] \circ f \circ g \circ \phi^{-1}$is $h \circ g \circ \phi^{-1}$. Thus, since $\phi$ is an analytic map from an open neighbourhood of the analytic curve $\Gamma''$ to an open neighbourhood of $\mathbb{S}^1$, 
  the Sobolev trace $\tilde{u}$ of $[\mathfrak{O}(\riem_1',\riem_2') (h \circ f^{-1})] \circ f \circ g$
  is $h \circ g$.  This proves (\ref{eq:step_one_final_proof}).  \\  
  
 Let $u$ be the harmonic Sobolev extension of $\tilde{u}$.  We claim that 
 \begin{equation} \label{eq:step_two_final_proof}
 u = [\mathfrak{O}(\riem_1,\riem_2) h ] \circ g.
 \end{equation}
 To see this, observe that 
 %both $u \circ \phi^{-1}$ and 
 $[\mathfrak{O}(\riem_1,\riem_2) h ] \circ g \circ \phi^{-1}$ has non-tangential limits on $\mathbb{S}^1$, except possibly on a null set, by Corollary \ref{co:annulus_nontangential}, and these are $h \circ g \circ \phi^{-1}$. Again by Lemma \ref{le:annulus_nontangential agrees with sobolev}, the Sobolev trace of $[\mathfrak{O}(\riem_1,\riem_2) h ] \circ g \circ \phi^{-1}$ thus equals $\tilde{u} \circ \phi^{-1}$.  
 Since $\phi$ is analytic on an open neighbourhood of $\Gamma''$, the Sobolev trace of $[\mathfrak{O}(\riem_1,\riem_2) h ] \circ g$ equals $\tilde{u}$. By uniqueness of the harmonic Sobolev extension on analytically bounded surfaces, $u = [\mathfrak{O}(\riem_1,\riem_2) h] \circ g$ as claimed.\\ 
 \\

 Recalling that $\tilde{u}$ is the Sobolev trace of $[\mathfrak{O}(\riem_1',\riem_2') (h \circ f^{-1})] \circ f \circ g$, by 
 (\ref{eq:step_two_final_proof}) and
 Dirichlet's principle as formulated in Lemma \ref{Dirichlets princip}
 we obtain
 \begin{equation} \label{eq:step_three_final_proof}
 \|  u \|_{\dot{H}^1 (\riem_2'')}  \leq \| [\mathfrak{O}(\riem_1',\riem_2') (h \circ f^{-1})] \circ f \circ g \|_{\dot{H}^1(\riem_2'')}. 
 \end{equation}
 Now Theorem \ref{equiv and sobolevs in riem}, and conformal invariance of Dirichlet energy yield that 
  \begin{align*}
   \| [\mathfrak{O}(\riem_1,\riem_2) h]  \|_{\mathcal{D}_\mathrm{harm}(\riem_2)} 
    & = \| u \|_{\mathcal{D}_\mathrm{harm}(\riem_2'')} \\
      & = \| u \|_{\dot{H}^1 (\riem_2'')} \\
    \ ^{(\ref{eq:step_three_final_proof})} & \leq \| [\mathfrak{O}(\riem_1',\riem_2') (h \circ f^{-1})] \circ f \circ g \|_{\dot{H}^1(\riem_2'')} \\
   & =  \| [\mathfrak{O}(\riem_1',\riem_2') (h \circ f^{-1})] \circ f \|_{\dot{H}^1(\riem_2)}. 
 \end{align*}
 Finally using quasi-invariance of Dirichlet energy and the estimate (\ref{eq:temp1_reflection}), there is a $K$ such that 
 \begin{align*}
   \| [\mathfrak{O}(\riem_1',\riem_2') (h \circ f^{-1})] \circ f \|_{\dot{H}^1(\riem_2)}  & \leq K \| [\mathfrak{O}(\riem_1',\riem_2') (h \circ f^{-1})]   \|_{\mathcal{D}_\mathrm{harm}(\riem_2')} \\
   & \leq K C \| h \|_{\mathcal{D}_\mathrm{harm}(\riem_1)}.   
 \end{align*}
 Combining the two previous estimates completes the proof.
  \end{proof}

\end{subsection}
\end{section}


\begin{thebibliography}{99}
 \bibitem{Ahlfors_Sario} Ahlfors, L. V.; Sario, L.
 {\it  Riemann surfaces}.
   Princeton Mathematical Series, No. 26 Princeton University Press, Princeton, N.J. 1960.
 \bibitem{ArcozziRochberg}   N. Arcozzi and R. Rochberg,
   {\it Invariance of capacity under quasisymmetric maps of the circle: an easy proof.} Trends in harmonic analysis, 27--32,  Springer INdAM Ser., {\bf 3}, Springer, Milan (2013).
\bibitem{Booss} Booss-Bavnbek, B.; Wojciechowski, K. P. {\it Elliptic boundary problems for Dirac operators.} Mathematics: Theory and Applications. Birkhäuser Boston, Inc., Boston, MA, 1993.
\bibitem{brewster_mitrea}Brewster, K.; Mitrea, D.; Mitrea, I.; Mitrea, M. {\it
Extending Sobolev functions with partially vanishing traces from locally $(\varepsilon,\delta)$-domains and applications to mixed boundary problems.} 
J. Funct. Anal. 266 (2014), no. 7, 4314--4421.
\bibitem{ConwayII}  Conway, J. B. {\it Functions of one complex variable }II. Graduate Texts in Mathematics, 159. Springer-Verlag, New York, 1995.
\bibitem{El-Fallah_etal_primer}  El-Falla, O. ; Kellay, K. ; Mashreghi, J. and Ransford, T. \textit{A primer on the Dirichlet space}.
Cambridge Tracts in Mathematics 203, Cambridge University Press, Cambridge (2014).

\bibitem{GauthierSharifi} Gauthier, P. M.; Sharifi, F. {\it The Carathéodory reflection principle and Osgood-Carathéodory theorem on Riemann surfaces}. Canad. Math. Bull. {\bf 59} (2016), no. 4, 776--793.
\bibitem{Helein}  H\'elein, F.; Wood, J. C. {\it{Harmonic maps}}. Handbook of global analysis, 417--491, 1213, Elsevier Sci. B. V., Amsterdam, 2008.

\bibitem{Johnsson}
Jonsson, A. {\it Besov spaces on closed subsets of $\mathbb{R}^n$.} Trans. Amer. Math. Soc. 341:1, 1994,
355--370.

\bibitem{Johnsson-Wallin} 
Jonsson, A., and Wallin, H: {\it Function spaces on subsets of $\mathbb{R}^n$.} Math. Rep. 2:1, 1984.

\bibitem{Jonesward} Jones, P. W.; Ward, L. A. {\it Fuchsian groups, quasiconformal groups, and conical limit sets}. Trans. Amer. Math. Soc. 352 (2000), no. 1, 311--362.
\bibitem{Kuramochi} Kuramochi, Z. {\it On Fatou's and Beurling's theorems}. Hokkaido Math. J. {\bf 11} (1982), no. 3, 262--278.
\bibitem{LehtoVirtanen_book} Lehto, O.; Virtanen, K. I. {\it Quasiconformal mappings in the plane}. Second edition. Translated from the German by K. W. Lucas. Die Grundlehren der mathematischen Wissenschaften, Band 126. Springer-Verlag, New York-Heidelberg, 1973.
\bibitem{Li} Li, P. Poincar\'e inequalities on Riemannian manifolds. Pages 73--85 in Seminar on Differential Geometry, Edited by Yau, S. T. pp. 669–706, Ann. of Math. Stud., 102, Princeton Univ. Press, Princeton, N.J., 1982. 
\bibitem{Mazya}  Mazya, V. \textit{Sobolev spaces with applications to elliptic partial differential equations.} Second, revised and augmented edition. Grundlehren der Mathematischen Wissenschaften [Fundamental Principles of Mathematical Sciences], 342. Springer, Heidelberg, 2011.
\bibitem{Medkova}
Medkov\'a, D. {\it The Laplace equation. Boundary value problems on bounded and unbounded Lipschitz domains.} Springer, Cham, 2018. 
 \bibitem{Neharibook}  Nehari, Z.  \textit{Conformal mapping}. Reprinting of the 1952 edition. Dover Publications, Inc., New York, 1975.
 \bibitem{Pommerenke_boundary_behaviour}  Pommerenke, Ch.  {\it Boundary behaviour of conformal maps.}
  Grundelehren der mathematischen Wissenschaften v 299, Springer-Verlag, 1991.
 \bibitem{Radnell_Schippers_monster} Radnell, D. and Schippers, E. \textit{Quasisymmetric sewing in rigged Teichmüller space}, Commun. Contemp. Math. {\bf{8}} (2006), no. 4, 481--534.
 \bibitem{RSS_Dirichlet_general} Radnell, D., Staubach, W., and Schippers, E. {\it Dirichlet
 spaces of domains bounded by quasicircles.}   arXiv:1705.01279v2 [math.CV] \bibitem{RSS_period_genus_zero}  Radnell, D.; Schippers, E.; and Staubach, W.  {\it Model of the Teichm\"uller space of genus zero by period maps}. Preprint.  	arXiv:1710.06960 [math.CV]
 \bibitem{RodinSario}  Rodin, B., and Sario, L.  {\it Principal functions.}  D. Van Nostrand Company, Inc.  1968.
 \bibitem{Royden}  Royden, H. L. {\it Function theory on compact Riemann surfaces}. J. Analyse Math. 18, 1967, 295--327.
 \bibitem{SarioNakai}  Sario, L. and Nakai, M.  {\it Classification theory of Riemann surfaces.} 
 Grundlehren der Mathematischen Wissenschaften {\bf 164} Springer-Verlag, 1970.
 \bibitem{SarioOikawa} Sario, L., and Oikawa, K. {\it Capacity functions}.  Grundlehren der Mathematischen Wissenschaften  {\bf 149}, Springer-Verlag, 1969.  
 \bibitem{Schiffer_Spencer} Schiffer, M. and Spencer, D.  {\it Functionals of finite Riemann surfaces}. Princeton University Press, Princeton, N. J., 1954.
\bibitem{SchippersStaubach_jump} Schippers, E. and Staubach, W.  {\it Harmonic reflection in quasicircles and well-posedness of a Riemann-Hilbert problem on quasidisks}. J. Math. Anal. Appl. 448 (2017), no. 2, 864--884.
\bibitem{SchippersStaubach_Grunsky_quasicircle} Schippers, E. and Staubach, W. {\it Riemann boundary value problem on quasidisks, Faber isomorphism and Grunsky operator}. Complex Anal. Oper. Theory 12 (2018), no. 2, 325–354. doi:10.1007/s11785-016-0598-4
 \bibitem{Seeley} Seeley, R. T. {\it Singular integrals on compact manifolds}. Amer. J. Math. 81, 1959, 658--690.
 \bibitem{Taylor}  Taylor, Michael E. {\it Partial differential equations} I. {\it Basic theory.} Second edition. Applied Mathematical Sciences, 115. Springer, New York, 2011.
\end{thebibliography}
\end{document}